\documentclass[11pt]{amsart}

\usepackage{epsfig}
\usepackage{graphicx}
\usepackage{amscd}
\usepackage{amsmath}
\usepackage{amsfonts}
\usepackage{mathrsfs}
\usepackage{amssymb}
\usepackage{verbatim}
\usepackage{comment}
\usepackage{color}
\usepackage{hyperref}

\newtheorem{theorem}{Theorem}[section]
\theoremstyle{plain}

\newtheorem{corollary}[theorem]{Corollary}

\newtheorem{defi}[theorem]{Definition}
\newtheorem{example}[theorem]{Example}

\newtheorem{lemma}[theorem]{Lemma}

\newtheorem{prop}[theorem]{Proposition}
\newtheorem{remark}[theorem]{Remark}

\numberwithin{equation}{section}

\def\what{\widehat}

\def\Xxi{{\mathfrak X}_\zeta}
\def\Xx{{\mathfrak X}}

\def\Onne{{1\!\!1}}
\def\One{{\bf 1}}

\def\dist{{\rm dist}}
\def\supp{{\rm supp}}

\def\wtil{\widetilde}
\def\half{\frac{1}{2}}

\newcommand{\lam}{\lambda}
\def\Lam{\Lambda}
\newcommand{\gam}{\gamma}
\newcommand{\om}{\omega}

\newcommand{\Sig}{\Sigma}
\def\Sigb{{\mathbf \Sig}}

\newcommand{\sig}{\sigma}

\def\bbe{{\bf e}}
\newcommand{\R}{{\mathbb R}}
\newcommand{\Q}{{\mathbb Q}}
\newcommand{\Z}{{\mathbb Z}}
\newcommand{\C}{{\mathbb C}}

\def\N{{\mathbb N}}

\newcommand{\Nat}{{\mathbb N}}

\def\wt{\widetilde}
\def\A{{\mathcal A}}

\def\Sk{{\mathscr S}}

\def\Pk{{\mathcal P}}

\def\Sf{{\sf S}}
\def\T{{\mathbb T}}

\def\bxi{{\mathbf \xi}}

\def\bz{{\mathbf z}}
\def\one{\vec{1}}
\def\be{\begin{equation}}
\def\ee{\end{equation}}
\newcommand{\Ek}{{\mathcal E}}

\newcommand{\eps}{{\varepsilon}}
\newcommand{\es}{\emptyset}

\def\ov{\overline}
\def\und{\underline}

\def\Span{{\rm Span}}
\def\Card{{\rm card}}

\def\Cc{{\mathscr M}}

\def\Mc{{\mathscr M}}

\def\a{a}

\def\nula{\nu_\lam}

\def\ve1{\vec{1}}

\def\M{{\mathbf M}}

\def\Ak{{\mathcal A}}
\def\Pbi{{\mathbf \Pi}}

\begin{document}

\author{Alexander I. Bufetov} 
\address{Alexander I. Bufetov\\ 
Aix-Marseille Universit{\'e}, CNRS, Centrale Marseille, I2M, UMR 7373\\
39 rue F. Joliot Curie Marseille France }
\address{Steklov  Mathematical Institute of RAS, Moscow}
\address{Institute for Information Transmission Problems, Moscow}
\email{alexander.bufetov@univ-amu.fr, bufetov@mi.ras.ru}
\author{Boris Solomyak }
\address{Boris Solomyak\\ Department of Mathematics,
Bar-Ilan University, Ramat-Gan, Israel}
\email{bsolom3@gmail.com}

\title{Self-similarity and spectral theory: on the spectrum of substitutions}

\begin{abstract}
In this  survey of the spectral properties of substitution dynamical systems we consider primitive aperiodic substitutions and associated dynamical systems: $\Z$-actions and $\R$-actions, the latter viewed as tiling flows. Our focus is on the continuous part of the spectrum.  For $\Z$-actions the maximal spectral type can be represented in terms of matrix
Riesz products, whereas for tiling flows,  the local dimension of the spectral measure is governed by the {\em spectral cocycle}. We give references to complete proofs and emphasize ideas and connections.
\end{abstract}

\maketitle

\pdfbookmark{Dedication}{dedication}
\thispagestyle{empty}
\begin{center}
  \Large \emph{Dedicated to Nikolai Kapitonovich Nikolski}
\end{center}

\section{Introduction}

Substitutions provide a rich class of examples of dynamical systems that are intermediate between periodic and random. The study of substitution sequences and associated structures has connections with
ergodic theory and harmonic analysis, number theory, especially Diophantine approximation,  theoretical computer science and mathematical physics, 
particularly in the study of quasicrystals. 
Substitutions have zero entropy and at most linear complexity; at the same time, the spectral properties of substitution systems are subtle and far from being {
 completely} understood.

\thispagestyle{empty}

\subsection{Substitutions: classical examples}
We start by recalling the main definitions and a few key examples. 
Let $\A$ be a finite alphabet and $\A^+$ the set of nonempty words with letters in $\A$. A {\em substitution}  is a map $\zeta:\,\A\to \A^+$, extended to 
 $\A^+$ and $\A^{\N}$ by concatenation. In order to explain how this works, consider
some of the best-known substitutions.

\begin{example} \label{ex1} {\em
 {\em Thue-Morse substitution:} $\zeta_{_{\rm TM}}(0)=01,\ \zeta_{_{\rm TM}}(1)=10$.
The first steps of the iteration are as follows:
$$0 \to 01\to 0110 \to 01101001 \to 01101001 1001 0110 \to \ldots $$
Notice that $\zeta^n_{_{\rm TM}}(0)$ starts with $\zeta^{n-1}_{_{\rm TM}}(0)$ for $n\ge 1$; thus
we obtain a well-defined limiting infinite sequence $u = \lim_{n\to \infty} \zeta_{_{\rm TM}}^n(0)$, which is a fixed point of the substitution: $\zeta_{_{\rm TM}}(u) = u=(u_n)_{n\ge 0}$.
It is called the Thue-Morse sequence or Prouhet-Thue-Morse sequence. There are many alternative definitions of $(u_n)$; for instance, 
$u_n$ is zero if and only if there are even number of 1's in the binary expansion of $n$.
For the rich history, the reader is referred to \cite{AlSh}. 
Here we mention a few highlights.

\begin{itemize}
\item The sequence $u$ implicitly appears in the 1851
work of Prouhet \cite{prouhet}  on a problem formulated in the correspondence between Leonhard Euler and Christian Goldbach and asking 
{for which $N$ and $k$ there is a partition of $\{0,\ldots,N-1\}$ into two disjoint subsets $S_0$ and $S_1$ such that $\sum_{j\in S_0} j^n = \sum_{j\in S_1} j^n$ for $n=0,\ldots,k$. 
Prouhet showed that
taking $N= 2^{k+1}$, with $S_0$ and $S_1$ being the places of 0's and 1's respectively in $\zeta_{_{\rm TM}}^{k+1}(0)$, yields a solution.}
  The general Prouhet-Tarry-Escott problem, as it is now called, is still open.

\item Thue (1906) \cite{Thue} introduced the sequence $u$ and showed that it is {\em cube-free} in the sense
that if $B$ is an arbitrary finite word, then the pattern $BBB$ does not appear in $u$.
Indeed, Thue established an even stronger claim: there is no word $B$ such that $BBb$ occurs in $u$, where
$b$ is the 1st letter of $B$.
Thue also noted that no sequence on 2 symbols is square-free, but on the 3-symbol alphabet $\{-1,0,1\}$ an
example of a square-free sequence is $v_n = u_{n+1}-u_n$. Thue introduces his studies by the words:
``For the development of logical sciences it will be important, without consideration for possible applications, to find large domains for speculation about difficult problems''
(translation of J. Berstel).

\item Morse (1921) \cite{Morse}  rediscovered the sequence in order to give an explicit construction of  a recurrent non-periodic geodesic on a surface of constant negative curvature. Morse notes also that the closure of the geodesic is not dense in the surface. Observe that, excluding degenerate cases,  any substitution would work for the same purpose.
The work of Morse was one of the first applications in geometry of {\em symbolic dynamics}: geodesics are coded  by symbolic sequences and the geometrical properties of the geodesics are derived from the combinatorial properties of the symbolic sequences.  

\item The study of {\em dynamical systems} acting on spaces of symbolic sequences, with substitutions providing a prominent example, seems to have originated in the work of Morse and Hedlund (1938) \cite{HM38}.

\item In subsequent work, Morse and Hedlund  (1944) \cite{HM44} 
used the Thue-Morse sequence this to construct an ``unending chess game'', if slightly modified rules for the draw are used.

\end{itemize}
}
\end{example}

\begin{example} \label{ex2} {\em
{\em Fibonacci substitution:} $\zeta_{_{\rm F}}(0) = 01,\ \zeta_{_{\rm F}}(1) = 0$. Iterating we obtain
$$0 \to 01 \to 010 \to 01001 \to \ldots \to u = 0100101001001010010100100\ldots,$$
the {\em Fibonacci substitution sequence}. Note that
$|\zeta_{_{\rm F}}^n(0)| = F_n$ are the Fibonacci numbers.}
\end{example}

\begin{itemize}
\item The Fibonacci substitution sequence is representative also of another important class in symbolic dynamics, namely, {\em Sturmian sequences}. They are characterized by the property of being non-periodic sequences of lowest possible complexity: the number of words of length $n$ is equal to $n+1$ for all $n$, whereas if there exists $n$ such  that the number of words of length $n$ is at most $n$, then the sequence is eventually periodic.

\item Sturmian sequences and the associated dynamical systems were 
investigated by Morse and Hedlund  \cite{HM40}, who proved, in particular, that they can be alternatively characterized as symbolic trajectories of an irrational rotation of the circle. The Fibonacci sequence arises from the coding of the rotation by the golden ratio: $x\mapsto x + \frac{\sqrt{5}-1}{2}$ mod 1.

\item It is known that a Sturmian sequence arises from a substitution if and only if the angle of rotation is a quadratic irrational; this is related to a theorem of
Lagrange, saying that a real number has a periodic continued fraction if and only if it is a quadratic irrational.
\end{itemize}

\begin{example} {\em 
{\em Rudin-Shapiro  substitution:} $\Ak = \{0,1,\ov{0},\ov{1}\}$.
$$
\zeta_{_{\rm RS}}:\ 0\mapsto 01,\ \ 1\mapsto 0\ov{1},\ \ \ov{0} \mapsto \ov{0}\,\ov{1},\ \ \ov{1}\mapsto\ov{ 0} 1
$$
The Rudin-Shapiro sequence is $(x_k) = \phi(u_k)$, where $u_k$ is the $k$-th letter of $u = \lim_{n\to \infty} \zeta_{_{\rm RS}}^n(0)$ and 
$\phi: 0, 1\mapsto +1,\ \ov{0},\ov{1} \mapsto -1$;
the first terms of the sequence are (keeping only the signs):
$
+++-++-++++---+-.
$
This sequence was introduced by Shapiro \cite{Shapiro} in his Thesis; its remarkable property is that
\be \label{ex-RS}
\forall\, N\ge 0,\ \ N^{1/2} = \Bigl\|\sum_{n< N} x_n e^{2\pi i n\om} \Bigr\|_2 \le \Bigl\|\sum_{n< N} x_n e^{2\pi i n\om} \Bigr\|_\infty \le CN^{1/2},
\ee
where the $L^2$, $L^\infty$ norms are on $[0,1]$ and one can take $C=2 + \sqrt{2}$. 
 J.-P. Allouche \cite{Allouche} pointed out that the sequence implicitly, and independently from Shapiro, appears in the work of Golay 
 \cite{Golay}; 
see also \cite{BriMor}. Rudin \cite{Rudin} published the proof of \eqref{ex-RS}, but he acknowledged Shapiro's priority; in fact, he was on the thesis committee of Shapiro!}
\end{example}

The substitution $\zeta$ is said to be of {\em constant length} $q$ if $|\zeta(a)|=q$ for all $a\in \A$, otherwise, it is of {\em non-constant length}. Thus, Thue-Morse and Rudin-Shapiro are constant
length substitutions, whereas Fibonacci is of non-constant length.

\subsection{Substitution dynamical systems and generalizations: historical remarks.} As already mentioned, symbolic dynamical systems associated with substitutions were first studied by Morse and Hedlund \cite{HM38} (see the next section for definitions). These are  examples of {\em discrete  topological dynamical systems}, that is, continuous maps on a compact metric space. When equipped with an invariant measure (which is unique under very general assumptions), we obtain
a measure-preserving system, and then the study of {\em spectral properties} of the associated unitary operator (the Koopman operator, see the next section for definition) becomes natural, as an important isomorphism invariant.

In some cases, the substitution system can be shown to arise from a coding of some other system of known structure: this usually provides complete spectral information. For instance, the Fibonacci substitution dynamical system is measurably (and even almost-topologically --- in fact, outside a countable
set of ``singularities'') isomorphic to an irrational rotation, hence has purely discrete spectrum. A large class of constant-$q$ length substitutions was shown
to provide a coding for the $q$-odometer, the translation by 1 map on the group of $q$-adic integers, hence again, has a purely discrete spectrum.
As will be discussed in the next section, deciding when the spectrum is discrete is solved in the constant length case, but is an important open problem
otherwise.
It is more difficult to find a model, or to identify the spectral properties, when the spectrum has a continuous component; the research in this direction is
on-going, and we will discuss some recent results in this paper.

Systematic investigation of topological substitution dynamical systems was started by Gottschalk (1963) \cite{Gott63} and continued in the 1960's-70's by the ``Rennes research group'', which included M. Keane, M. Dekking, and P. Michel, and also by E. Coven,  J. Martin, and T. Kamae, who worked both in the topological and in the measurable category. 
{A representative, but incomplete list of references from this period is \cite{keane,Kamae,CovenKeane,Martin1,Martin2,Michel,Michel2,Dek1,DK}.}

In the 1970's-80's, pioneering work was done by G\'erard Rauzy in Marseille. As had earlier been the case with Prouhet,  Rauzy's work had number-theoretic questions at the source. An extremely fruitful idea of ``inducing'', or
``renormalization'', developed in \cite{Rauzy1} was to extend the continued fraction algorithm to the setting of interval exchange transformations (IET's).
As mentioned above, the Fibonacci substitution dynamical systems provides a coding of the circle rotation by the golden ratio. For a general irrational
rotation, the regular continued fraction provides a {\em sequence of substitutions}, which generates the associated Sturmian symbolic system.
An exchange of $m$ intervals is a piecewise isometry, determined by a permutation of $\{1,\ldots,m\}$ and
a vector $(\lam_1,\ldots,\lam_m)$, with $m\ge 2$ and all $\lam_j$ positive. The IET ``chops up'' the line segment $[0,\sum_{j=1}^m \lam_j)$ into subintervals of length $\lam_j$, which are then reshuffled according to the permutation. Thus, a 2-interval exchange is isomorphic to a circle rotation.
A meeting with G. Rauzy in Rennes in 1977 (see \cite{FMS19}) inspired the fundamental work of William Veech on interval exchanges and the Teichm\"uller flow \cite{veech}. 
 The Rauzy-Veech formalism codes an interval exchange transformation by a sequence of substitutions; 
 the substituting morphism is no longer fixed but evolves over time; this evolution  is itself described by a stationary process --- an invariant measure for the Rauzy-Veech renormalization dynamical system. 
 
Following S. Ito \cite{Ito}, A. M. Vershik \cite{Vershik} introduced in 1982 the notion of {\em adic transformation} on an {\em ordered Bratteli diagram}, or, in Vershik's terminology, a Markov compactum. This construction admits various degrees of generality: if sufficiently general Markov compacta are used, then any ergodic measure-preserving transformation can be so encoded. For a survey and developments of the Vershik  formalism for maps see e.g., \cite{VerLiv,DHS}.
If two-sided Markov compacta are used, then one obtains an adic coding for  measure-preserving flows: for example, translation flows on flat surfaces are obtained as  bi-infinite sequences of substitutions governed by a stationary process, which, in turn, can be identified with a probability measure, invariant and ergodic under the action of the Teichm\"uller flow,  on a connected component of a stratum of the moduli space of abelian differentials.  This formalism is used in \cite{Buf-umn,Bufetov1} in  studying the deviation of ergodic averages for translation flows.

Ya. G. \ Sinai asked [personal communication] about the H\"older regularity of the spectrum of interval exchange transformations. In \cite{BuSo14,BuSo18,BuSo20,BuSo21},  H\"older regularity is established for almost every  translation flow on a flat surface of genus $\ge 2$. 
Forni \cite{Forni},  motivated by \cite{BuSo18,BuSo20}, which had established the result in the case of genus 2, obtained the result for an arbitrary genus $\ge 2$ earlier than us, by
a different, albeit related, technique. Very recently the H\"older regularity of the spectrum
 for almost every non-rotation interval exchange transformation was obtained in \cite{AFS}. 
 Observe that while in this case the results in discrete and continuous time are similar, 
  there is no a priori relation between the spectrum  of a transformation and a suspension flow over it.

In this survey, we restrict ourselves to  substitution systems and suspension flows  over them,  
but many results can be extended to the  $S$-adic case.

\subsection{Tilings, quasicrystals, and diffraction spectrum} Rather unexpectedly, substitutions became popular in mathematical physics, in the study of {\em quasicrystals}.
The fascinating story of their discovery by Dan Schechtman in 1982, who got the Nobel Prize in Chemistry for this achievement in 2011, is widely known, see e.g., \cite{Sen,BaGr}.
Citing Senechal's book \cite{Sen}, ``it was a demise of a paradigm'' which claimed that ``a crystalline structure is {\em by definition} a periodically repeating pattern''. With the discovery of $X$-ray diffraction in 1912,
such a structure was deduced from a diffraction pattern, and it was universally accepted that a solid is a crystal if and only if the diffraction is discrete, i.e., a regular pattern of sharp bright spots, the so-called ``Bragg peaks''.
 Schechtman discovered that a certain aluminium-manganese alloy produces a discrete diffraction pattern, which has a five-fold rotational symmetry, incompatible with periodicity.
This caused a big ``splash'' in the scientific community: physicists, mathematicians, and material scientists worked hard to understand this phenomenon. Amazingly, a mathematical model already existed, namely, the {\em Penrose tilings}, invented by Roger Penrose in 1976. The main feature of the Penrose tiling is that it has a small set of basic tiles (just two in some versions), which can
tile the plane, but only non-periodically. In a special case (there are, in fact, uncountable many distinct Penrose tilings) five-fold rotational symmetry is observed, and {\em statistical} 5-fold
symmetry is present in all Penrose tilings.
Independently of Schechtman's discovery, the crystallographer A. Mackay showed experimentally that the Penrose tiling produces a discrete diffraction pattern with a five-fold rotational symmetry, which later turned out to match well the diffraction pattern of Schechtman's quasicrystals. 

Penrose tilings have many special features: they can be defined by three completely different ways: ``local rules'' (like a ``jigsaw puzzle''), ``substitution'', and ``projection method'',
see \cite{BaGr}. The substitution is  2-dimensional and ``geometric'': the role of symbols is played by the basic tiles, and the operation of substitution
proceeds by expanding by homothety and then subdividing. More general  substitution tilings consequently became popular as well, as models of ``aperiodic order''. Their diffraction spectrum is closely related to dynamical spectrum, {see \cite{Dworkin,Hof95,Hof97}; we outline this connection in Section~\ref{sec:diffract}.} See \cite{RW,Robi1,SolTil} for some
early work on tilings, dynamics, and their spectrum; see also the survey \cite{Robi}. For a very general introduction, see \cite{Radin} and for the general subject of aperiodic order  see \cite{BaGr} and references therein.

\medskip

\noindent {\bf {Acknowledgements.} } {We are grateful to Jean-Paul Allouche and to Houcein El Abdalaoui for very helpful comments on a preliminary version of the manuscript. }A.\,B.'s  research received support from the European Research Council (ERC) under the European Union Horizon 2020 research and innovation programme, grant 647133 (ICHAOS) and from the Agence Nationale de la Recherche, project ANR-18-CE40-0035. The research of B.\,S.\ was supported by the Israel Science Foundation
(grant 911/19).

We are happy to dedicate this article to Nikolai Kapitonovich, advisor and mentor of B.\,S.\ during 1978-1988, to whom B.\,S. owes his love of spectral theory.


\section{Background}

Consider the space of bi-infinite sequences $x=(x_j)_{j\in \Z}\in \A^\Z$, equipped with the product topology induced by a metric, which is defined for $x\ne y$ by
$$
\varrho(x,y) = 2^{-n},\ \ \mbox{where}\ \ n = \min\{|j|: x_j \ne y_j\}.
$$
It is a compact metric space, homeomorphic to the Cantor set. Let $\zeta$ be a substitution on $\A$.
The {\em substitution space} $X_\zeta$ is defined as the set of all bi-infinite sequences $x\in \A^\Z$ such that any word  in $x$
appears as a subword of $\zeta^n(\a)$ for some $\a\in \A$ and $n\in \N$. It is clearly shift-invariant, and we obtain a (topological) {\em substitution dynamical system}  $(X_\zeta, T)$, where
$T$ is the restriction of the left shift from $\A^\Z$. Sometimes the one-sided substitution space $X^+_\zeta\subset \A^\N$ is also considered, on which the left shift $T_+$ is onto, but not 1-to-1.
One should be careful to distinguish the shift action $(X^+_\zeta,T_+)$ from the substitution action $(X^+_\zeta,\zeta)$.
Passing to a power of $\zeta$ if necessary (this does not affect the substitution space), one can assume that there exists a letter $a$ such that $\zeta(a)$ starts with $a$. We then obtain a one-sided fixed point of the substitution action:
\be \label{fixpt}
u=u_0u_1u_2\ldots = \lim_{n\to\infty} \zeta^n(a).
\ee

The {\em substitution matrix} $\Sf=\Sf_\zeta=(\Sf(i,j))$ is the $d\times d$ matrix, where $d=\# \A$, such that $\Sf(i,j)$ is the number
of symbols $i$ in $\zeta(j)$. In our examples we have
$$
{\rm (i)}\ \ \Sf_{\zeta_{_{\rm TM}}} = \left[ \begin{array}{cc} 1 & 1 \\ 1 & 1 \end{array} \right],\ \ {\rm (ii)}\ \ \ \ \ \Sf_{\zeta_{_{\rm F}}} = \left[ \begin{array}{cc} 1 & 1 \\ 1 & 0 \end{array} \right],\ \ \ \ \ 
{\rm (iii)}\ \ \Sf_{\zeta_{_{\rm RS}}} = \left[ \begin{array}{cccc} 1 & 1 & 0 & 0 \\ 1 & 0 & 0 & 1 \\ 0 & 0 & 1& 1 \\ 0 & 1 & 1 & 0  \end{array} \right].
$$

We will always assume that the substitution is {\em primitive}, that is, ${(\Sf_\zeta)}^n$ is strictly positive (entrywise) for some $n\in \Nat$.
Equivalently, starting from any symbol in $\A$, iterating the substitution we eventually obtain a word with all symbols present. 

\begin{prop}[{see \cite[Prop.\,5.5]{Queff}}]
Given a primitive substitution $\zeta$, the associated dynamical system is {\em minimal}, that is, $X_\zeta$ contains no proper closed shift-invariant subsets.
\end{prop}

If $\zeta$ is primitive, the Perron-Frobenius (PF) Theorem says that $\Sf$ has a unique dominant (PF) eigenvalue $\theta>1$, 
strictly greater than all other eigenvalues in modulus, and the corresponding right and left eigenvectors are strictly positive. 
One can show that every symbol  occurs in $x\in X_\zeta$ with a uniform frequency and these frequencies form a right PF eigenvector for $\Sf_\zeta$.

Another important assumption is that the substitution is  {\em non-periodic}, that is, $X_\zeta$ contains no periodic sequences.
In the primitive case this is equivalent to the space $X_\zeta$
 being infinite. Usually it is not hard to check aperiodicity: in the primitive constant-length case a criterion due to Pansiot \cite{Pansiot} says that $\zeta$ is non-periodic if and only if there exists
 $\alpha\in \A$ which occurs in $x\in X_\zeta$ with at least two distinct neighborhoods $\gam\alpha\delta$. In the non-constant length case a sufficient condition for non-periodicity is that
 the PF eigenvalue of $\Sf_\zeta$ is irrational. In the 2-symbol case a complete characterization of non-periodic substitutions was recently given in \cite{MRST}, but for larger alphabets a general
 criterion seems to be lacking.

\subsection{Invariant measure and spectrum.} The Borel $\sig$-algebra on $\A^\Z$ is generated by {\em cylinder sets}: for $w\in \A^n$ we denote $[w]_k:= \{x\in \A^\Z:\ u_{k}\ldots u_{k+n-1} = w_1\ldots w_n\}$ and write $[w]= [w]_0$. The following is due to Michel \cite{Michel}, see also \cite[Th.\,5.6]{Queff}.

\begin{theorem} \label{th-ue}
If $\zeta$ is a primitive substitution, then the dynamical system $(X_\zeta,T)$ is uniquely ergodic, that is, it has is a unique invariant Borel probability measure $\mu$.
\end{theorem}

In fact, one shows that for $w\in \A^+$ the measure of the cylinder set $\mu([w]_k)$ (independent of $k$ by shift invariance) is the frequency of occurrences of $w$ in $x\in X_\zeta$, which
exists uniformly, in view of the Perron-Frobenius Theorem. In particular, 
\be \label{freqPF}
\mbox{\em $(\mu([j]))_{j\in \A}$\ \  is a right PF eigenvector of \  $\Sf_\zeta$.}
\ee

\smallskip

Given a measure-preserving system $(X,T,\mu)$ (for us the measure is always finite, by default), it is of interest to study the spectral properties of the associated {\em Koopman operator}
$U= U_T:\ f\mapsto f\circ T$  on $L^2(X,\mu)$. Measure-preserving means that $T_*\mu = \mu$. If $T$ is $\mu$-preserving and a.e.\ invertible, it is  called an {\em automorphism} of the 
measure space $(X,\mu)$;
this implies that $U$ is a unitary operator.
Observe that $U \One = \One$, where $\One$ is the constant-one
 function in $L^2(X,\mu)$. Thus, the ``interesting part'' of spectral theory occurs on $L^2_0 =\{f\in L^2: \langle \One,f\rangle = 0\}$.
The symbol $\sig_U$ denotes the usual (compact) spectrum of the unitary operator.
Recall that for $f,g\in L^2(X,\mu)$ the (complex) spectral measure $\sig_{f,g}$ is determined by the equations
$$\widehat{\sig}_{f,g}(-k) =\int_0^1 e^{2\pi i k \om}\,d\sig_{f,g}(\om)=
\langle U^k f,g\rangle,\ \ k\in \Z.
$$
We write $\sig_f = \sig_{f,f}$.

Suppose we are given two automorphisms $(X_1,T_1,\mu_1)$, $(X_2,T_2,\mu_2)$ and a map $\varphi:X_1\to X_2$ such that $\varphi_*\mu_1 = \mu_2$ and
$\varphi\circ T_1 = T_2 \circ\varphi$. We say that $(X_2,T_2,\mu_2)$ is a {\em factor} of $(X_1,T_1,\mu_1)$ if $\varphi$ is onto and that they are {\em isomorphic} if $\varphi$ is
invertible (as usual, all such statements and equalities are understood a.e.). For isomorphic automorphisms the corresponding Koopman operators are unitarily equivalent; if $T_2$ is a factor of
$T_1$, then $U_{T_2}$ is unitarily equivalent to restriction of $U_{T_1}$ on an invariant subspace of $L^2(X_1,\mu_1)$.
 For the background in Ergodic Theory and related Spectral Theory the reader is referred to \cite{KSF,Walters,EW}; we recall some basic facts:

\begin{theorem} \label{th-e1}
Let $(X,T,\mu)$ be an automorphism.

{\rm (i)} The system is ergodic $\iff$ $\lam = 1$ is a simple eigenvalue for $U_T$.

{\rm (ii)} The system is weakly mixing $\iff$ $U_T|_{L^2_0}$ has no eigenvalues.

{\rm (iii)} The system is (strongly) mixing $\iff$  for any $f\in L^2_0(X,\mu)$ we have $\widehat{\sig}_f (k) \to 0$ as $k\to \infty$.
\end{theorem}

For an ergodic measure-preserving system, every eigenfunction has constant modulus a.e., hence eigenfunctions are actually in $L^\infty$ and as a consequence, the set of eigenvalues forms
a subgroup of $\T$. (In general, one can usually restrict considerations to ergodic systems using the {\em ergodic decomposition}, but this will not be our concern since primitive substitution systems
are uniquely ergodic, hence ergodic.)
We  recall the following classical 

\begin{theorem}[Halmos and von Neumann]
{\rm (i)} Two ergodic measure-preserving system with purely discrete spectrum are measurably isomorphic if and only if they have the same group of eigenvalues.

{\rm (ii)} An ergodic measure-preserving system with purely discrete spectrum is measurably isomorphic to a translation on a compact abelian group.
\end{theorem}

Spectral properties of substitution dynamical systems have been studied extensively, see \cite{Queff,Siegel} and references therein. These systems are never strongly mixing by a result of
Dekking and Keane \cite{DK}, hence there is always a singular spectral component, by Theorem~\ref{th-e1}(iii) and the Riemann-Lebesgue Lemma.
 Before going further, let us demonstrate some possibilities for the spectrum of substitution systems on examples given above.

\medskip

(i) The Thue-Morse substitution automorphism has purely  singular spectrum. It factors onto the translation $x\mapsto x+1$ on the group $\Z_2$ of
2-adic integers; the discrete component of the spectrum
corresponds to the restriction of $U_T$ to the space of $L^2$ functions that are even with respect to the involution $\ov{0}=1$, $\ov{1}=0$. On the orthogonal complement -- the subspace of odd functions -- the spectrum is singular continuous, given by a Riesz product. We will describe it in more detail and sketch a proof of singularity below.

(ii) The Fibonacci substitution system has purely discrete spectrum; it is isomorphic to the translation on the 1-torus $\T^1$: $x\mapsto x + \frac{\sqrt{5}-1}{2}$ (mod 1).

(iii) The Rudin-Shapiro substitution has a mixed spectrum: the discrete component corresponds to the factor onto the translation $x\mapsto x+1$ on the group $\Z_4$ of 4-adic integers. 
There is also a Lebesgue component of multiplicity 2.

\medskip

Here are two general facts about the spectrum as a {\em set}. The automorphism $T$ is called {\em aperiodic} if the set of $T$-periodic points has measure zero.

\begin{theorem}[Nadkarni \cite{Nad}] \label{th-dense}
If $T$ is an aperiodic automorphism of a probability measure space $(X,\mu)$, then $\sig_U = \T$ for $U=U_T$.
\end{theorem}

It is a more delicate problem to determine the support of the spectral measure $\sig_f$ for a specific test function $f$. Here is a recent result:

\begin{theorem}[A. Borichev, M. Sodin, B. Weiss \cite{BSW}]
If $T$ is an aperiodic automorphism of a probability measure space $(X,\mu)$ and $f:X\to \Z$ is finite-valued then either $\supp(\sig_f) =\{e^{2\pi i k/N}\}_{k\le N}$ for some $N\in \N$, or
$\supp(\sig_f) = \T$.
\end{theorem}

\subsection{Discrete component, purely discrete spectrum} This topic is not the focus of our paper, but we cannot omit it entirely, so we give a brief overview below.
\subsubsection{Substitutions of constant length}
Suppose that the substitution has constant length $q$. Then the group of eigenvalues is non-trivial and contains
the group of $q$-adic rationals; a complete description of the discrete spectral component was given by Dekking \cite{Dek1}. 
The {\em height} $h$ of the substitution is defined as follows: let
$$
g_\ell ={\rm gcd} \{k\ge 1:\ u_{k+\ell}=u_\ell\},
$$
where $u=(u_j)_{j\ge 0}$ is the fixed point of the substitution (\ref{fixpt}); then define, for a substitution $\zeta$ of constant length $q$:
\be \label{height}
h=h(\zeta) = \max\{n\ge 1:\ (n,q)=1,\ n\ \mbox{divides}\ g_0\}.
\ee
In turns out that also $h(\zeta) = \max\{n\ge 1:\ (n,q)=1,\ n\ \mbox{divides}\ g_\ell\}$ for  any $\ell >0$.
In particular, $h=1$ if there is any word of the form $\alpha\alpha$ in $x\in X_\zeta$.

Dekking \cite{Dek1} proved that the group of eigenvalues for $(X_\zeta,T,\mu)$, where $\zeta$ is primitive aperiodic of constant length $q$, is $e(\Z(q)\times \Z/h\Z)$, where
$\Z(q)$ is the group of $q$-adic rationals and $e(t) = e^{2\pi i t}$. 
Furthermore, Dekking showed that it is always possible to reduce a substitution with $h>1$ to one with $h=1$, called ``pure base''.
For a  constant-length substitution of height one, 
{\em Dekking's coincidence condition} \cite{Dek1}  gives an answer for when the spectrum is purely discrete. 

\begin{theorem}[Dekking 1978] \label{thm:dekking}
A primitive aperiodic
substitution, of constant length $q$ and height one,
has purely discrete spectrum if and only if $\zeta$ has a {\em coincidence}, namely, there exist $k$ and
$n$ such that
$$
(\zeta^k(0))_n =\ldots = (\zeta^k(d-1))_n
$$
\end{theorem}

\begin{example}
{\em Period-doubling substitution:} {\em $\zeta(0) = 01, \zeta(1) = 00$. There is a coincidence already for $k=1$, so the spectrum is purely discrete; in fact, the substitution automoprhism is 
isomorphic to the shift on the group of 2-adic integers.}
\end{example}

\begin{defi} \label{def:bijective}
{\em A substitution $\zeta$ of constant length $q$ on an alphabet $\Ak$ is called {\em bijective} if there are $q$ permutations of $\Ak$, denoted $\pi_1,\ldots,\pi_q$, such that
$$
\zeta(\alpha) = \pi_1(\alpha)\ldots\pi_q(\alpha),\ \ \alpha\in \A.
$$
{ If, in addition, the subgroup of the symmetric group ${\frak S}_m$, where $m=\#\Ak$,
 generated by the permutations $\pi_j$ is Abelian, the substitution is called {\em Abelian bijective}.}
}
\end{defi}

{ Of course, all bijective substitutions on 2 symbols are Abelian bijective, but this is rare for $m>2$.}
For a bijective substitution there is {\em never} a coincidence, hence the spectrum is mixed. The Thue-Morse substitution is bijective, another example is
$\zeta:\ 0 \mapsto 010,\ 1\mapsto 122,\ 2\mapsto 201$.

\subsubsection{Substitutions of non-constant length}

Non-constant length substitutions may be weakly mixing, i.e., have no non-trivial discrete component of the spectrum; 
this mostly depends on  algebraic and number-theoretic properties of the substitution matrix, see \cite{Host,Liv1,FMN}.
The general criterion is too involved to state here; instead we discuss the best known special case.

\begin{defi} \label{def-Pisot}
A substitution $\zeta$ is called an {\em irreducible Pisot substitution} if the characteristic polynomial of $\Sf_\zeta$ is irreducible over $\Q$ and all the eigenvalues of $\Sf_\zeta$, except the dominant one, lie inside the unit circle. In particular, this dominant PF eigenvalue is a Pisot  number, that is, an algebraic integer whose Galois 
conjugates are all less than one in modulus.
\end{defi}

The best known examples are the Fibonacci substitution from Example~\ref{ex2} and more generally, the so-called ``multinacci substitutions'' $0\to 01,\ 1\to 02,\ldots,m-1\to 0m,\ m\to 0$, for
$m\ge 2$,
{ in which case the substitution automorphism is isomorphic to an irrational translation on the $(m-1)$-dimensional torus \cite{SolAdic,SolSubs}}.

An irreducible Pisot substitution automorphism on $m$ symbols factors onto an irrational translation on the torus $\T^{m-1}$, so it has a large discrete component. The famous {\em Pisot discrete
spectrum conjecture} claims that the spectrum is purely discrete; it is still open as of this writing. Among early contributors (and apparently, authors of the conjecture) are G. Rauzy \cite{Rauzy2},
B. Host \cite{Host}, and A. N. Livshits (independently, in the framework of adic transformations). 
Livshits \cite{Liv0,Liv2} introduced a {\em balanced block algorithm} which made it possible to verify the conjecture in many cases.
It was completely settled only for 2 symbols \cite{BD,HS}. For more details on this, see \cite{ABBLS}.


\section{Spectral measures as Riesz products} \label{sec:Riesz}
The first part of this section is based on \cite{BuSo14}.
Our aim  is to represent the spectral measure of a substitution dynamical system via a matrix analog of Riesz products.
Various classes of generalized Riesz products have
appeared in the spectral theory of measure-preserving transformations, see e.g.\ \cite{Led,Bourgain1,ChoNa,KleRe,DoEig,Abda}, 
but for substitution dynamical systems it becomes necessary to consider {\it matrix analogues } of Riesz products. This was done by Queffelec \cite[Chapter 8]{Queff} for constant-length substitutions; here we consider the general case.
The following is standard.

\begin{lemma}\label{lem-spec1} Let $T$ be an automorphism of a measure space $(X,T,\mu)$ and $U= U_T$.
For any $f,g\in L^2(X,\mu)$ we have
$$
\sig_{f,g}= \mbox{\rm weak*-}\lim_{N\to \infty} \frac{1}{N} \left\langle \sum_{n=0}^{N-1} e^{-2 \pi i n \om}U^n f, \sum_{n=0}^{N-1} e^{-2 \pi i n \om}U^n g\right\rangle\,d\om,
$$
where in the right-hand side we consider the weak*-limit of absolutely continuous measures with the given density.
\end{lemma}

\begin{proof}
It is enough to check that the Fourier coefficients of the absolutely continuous measures in the right-hand side converge to $\widehat{\sig}_{f,g}(k)$ for all $k\in\Z$. Note that
$$\langle e^{-2\pi i n \om}U^n f, e^{-2\pi i \ell\om}  U^\ell g\rangle = \langle U^{n-\ell} f, g \rangle e^{-2\pi i (n-\ell)\om},$$ and $(e^{-2\pi i n\om} d\om)\,\widehat{\ }\,(-k) = \delta_{k,n}$, and the claim follows easily.
\end{proof}

The lemma highlights the role of {\em twisted Birkhoff sums}: for $x\in X$ and $f\in L^2(X,\mu)$ let
\be \label{def-SN}
S_N^x(f,\om) = \sum_{n=0}^{N-1}e^{-2\pi i n\om} f(T^n x),
\ee
and
\be \label{def-G}
G_N(f,\om) = N^{-1} {\left\| S_N^x(f,\om)\right\|}_2^2,
\ee
and note that 
$
\sig_f =  \mbox{\rm weak*-}\lim_{N\to \infty} G_N(f,\om)\, d\om
$
by Lemma~\ref{lem-spec1}. A direct computation shows that
\be \label{Fejer}
G_N(f,\om) = \int_0^1 K_{N-1}(\om-\tau)\,d\sig_f(\tau),\ \ \mbox{where}\ \ K_{N-1}(\om) = \frac{1}{N} \Bigl(\frac{\sin(N\pi \om)}{\sin \pi \om}\Bigr)^2
\ee
is the Fej\'er kernel.
We would like to relate the growth of $G_N(f,\om)$ as $N\to \infty$ to the local spectral properties of $\sig_f$ at $\om$. 
Based on a heuristic argument, it was suggested in the physics literature, e.g., \cite{Aubry},
that the scaling properties of $G_N(f,\om)$ can detect singular continuous spectrum. Some rigorous results in this direction
were obtained by Hof \cite{Hof}.
Notice that
$
\sig_f(\{\om\}) = \lim_{N\to \infty} N^{-1} G_N(f,\om)
$
by Wiener's Lemma (see e.g., \cite[p.42]{Katznelson}) and $\lim_{N\to \infty} G_N(f,\om) = \frac{d\sig_f}{d\om}$ Lebesgue-a.e., 
where $\frac{d\sig_f}{d\om}$ is the density of the absolutely continuous part of
$\sig_f$. In the intermediate range of growth the following elementary estimate can be useful, which is immediate from \eqref{Fejer}.

\begin{lemma} [{\cite{Hof,BuSo14}}]
For all $\om\in[0,1)$ and $r\in (0,\half]$,
\be \label{eq-estim1}
\sig_f(B(\om,r))\le \frac{\pi^2}{4N} G_N(f,\om),\ \ \ \mbox{where}\ \ N = \lfloor (2r)^{-1}\rfloor.
\ee
\end{lemma}

\medskip

We will return to similar estimates in the following sections, but now let us
restrict ourselves to substitution automorphisms, where we can benefit from unique ergodicity.
If $X$ is a metric space and $(X,T,\mu)$ is uniquely ergodic, then for all $f,g\in C(X)$ and all $k\in \Z$:
\be
\widehat{\sig}_{f,g}(-k)=\langle U^k f,g\rangle =  \int_X f(T^k x) \ov{g(x)}\,d\mu(x) 
                                                                                                    = \lim_{N\to \infty} \frac{1}{N} \sum_{n=0}^{N-1} 
                                                                                                   f(T^{n+k} x)\ov{g(T^n x)} \label{unerg1},
\ee
where $x\in X$ is arbitrary and the limit is uniform in $x$, see, e.g., \cite[Th.\,4.10]{EW}.

Let $\zeta$ be a primitive aperiodic substitution.
A function $f$ on $X_\zeta$ is  said to be {\em cylindrical} if it depends only on $x_0$, the 0-th term of the sequence $x\in X_\zeta$, that is, 
for some $\varphi:\A\to \C$ we have $f=f_\varphi(x) = \varphi(x_0)$.
Cylindrical functions form a $d$-dimensional vector space, with a basis
$\{\Onne_{[a]}: \, a\in \Ak\}$. Denote
$$
\sig_\a:= \sig_{\Onne_{[\a]}}\ \ \ \mbox{and}\ \ \ \sig_{ab}:= \sig_{\Onne_{[a]},\Onne_{[b]}}.
$$
These measures are also known as {\em correlation measures}. 
In view of (\ref{unerg1}), since primitive substitution dynamical systems are uniquely ergodic,
\be \label{corr1}
\widehat{\sig_{a,b}}(-k) = \lim_{N\to \infty} \frac{1}{|\zeta^N(\gam)|} \,\Card\left\{0 \le n+k < |\zeta^N(\gam)|:\ \zeta^N(\gam)_{n+k} = \a,\ \zeta^N(\gam)_n = b\right\}
\ee
for any $\gam\in \A$. For a word $v= v_0 v_1\ldots\in \A^+$ let
\be \label{def-Phi}
\Phi_\varphi(v,\om) = \sum_{j=0}^{|v|-1} \varphi(v_j) e^{-2\pi i  \om j}.
\ee
We will also write
$
\Phi_\a(v,\om) = \Phi_{\delta_a},
$
where $\delta_a(v_j) = 1$ if $v_j=a$ and  $\delta_a(v_j) = 0$ otherwise. Notice that $\Phi_\varphi$ is a twisted Birkhoff sum of a cylindrical function:
$$
S_N^x(f_\varphi,\om) = \Phi_\varphi\bigl(x[0,N-1],\om\bigr).
$$
It is immediate from (\ref{corr1}), as in Lemma~\ref{lem-spec1}, that
\be \label{corr2}
\sig_{a,b} = \mbox{\rm weak*-}\lim_{N\to \infty} \frac{1}{|\zeta^N(\gam)|}\,{\Phi_a(\zeta^N(\gam),\om)}\cdot\ov{\Phi_b(\zeta^N(\gam),\om)}\,d\om
\ee
and
\be \label{corr222}
\sig_{f_\varphi} = \mbox{\rm weak*-}\lim_{N\to \infty} \frac{1}{|\zeta^N(\gam)|}\,\left|\Phi_\varphi(\zeta^N(\gam),\om)\right|^2d\om.
\ee

\smallskip

Before analyzing the general case, let us consider the Thue-Morse substitution $\zeta$ to get a general idea. Recall that $\zeta_{_{\rm TM}}(0) = 01,\ \zeta_{_{\rm TM}}(1)=10$, then
 $\zeta_{_{\rm TM}}^n(0) = \zeta_{_{\rm TM}}^{n-1}(0)\,\ov{\zeta_{_{\rm TM}}^{n-1}(0)}$ for $n\ge 1$, where we use the convention $\ov{0}=1$, $\ov{1}=0$.
 Let $\varphi(0) = 1, \varphi(1) = -1$. Then we obtain by induction
$$
\Phi_\varphi(\zeta_{_{\rm TM}}^{N+1}(0),\om) = \Phi_\varphi(\zeta_{_{\rm TM}}^N(0),\om) \cdot (1 - e^{-2\pi i \om 2^N}) = 
                                \cdots        = \prod_{n=0}^{N} (1 -  e^{-2\pi i \om 2^n})
$$
Thus, for the Thue-Morse automorphism, the spectral measure is a generalized Riesz product
\be \label{eq-TM1}
\sig_{f_\varphi} = \mbox{\rm weak*-}\lim_{N\to \infty} 2^{-N}\Bigl|\prod_{n=0}^{N-1} (1 -  e^{-2\pi i \om 2^n})\Bigr|^2d\om.
\ee
 
 \smallskip

Returning to the general case, note that for any two words $u,v$ and the concatenated word $uv$,
 \be \label{eq-Phi}
 \Phi_\a(uv,\om) = \Phi_\a(u,\om) + e^{-2\pi i \om |u|} \Phi_\a(v,\om).
 \ee
Suppose that
 $\zeta(b) = u_1^{(b)} \ldots u_{k_b}^{(b)}$ for $b\in \A$. Then $\zeta^{n}(b)= \zeta^{n-1}(u_1^{(b)})\ldots \zeta^{n-1}(u_{k_b}^{(b)})$ for $n\ge 1$, hence (\ref{eq-Phi}) implies for all $b\in \A$:
$$
\Phi_a(\zeta^{n}(b),\om) = \sum_{j=1}^{k_b} \exp\left[-2\pi i \om \left(|\zeta^{n-1}(u_1^{(b)})| + \cdots +|\zeta^{n-1}(u_{j-1}^{(b)})|\right)\right] \Phi_a(\zeta^{n-1}(u_j^{(b)}),\om)
$$
(if $j=1$, the expression reduces to $\exp(0)=1$ by definition).
Let
\be \label{def-Psi}
\vec{\Psi}^{(a)}_{n}(\om):=\left( \begin{array}{c} \Phi_a(\zeta^{n}(1),\om) \\ \vdots \\ \Phi_a(\zeta^{n}(m),\om) \end{array} \right)\ \ \ \mbox{and}\ \ \
\Pbi_n(\om) = [\vec{\Psi}_n^{(1)}(\om),\ldots, \vec{\Psi}_n^{(m)}(\om)],
\ee
where $\Pbi_n(\om)$ is the $m\times m$ matrix-function specified by its column vectors.
It follows that
\be \label{eq-matr1}
\Pbi_n(\om)=\M_{n-1}(\om) \Pbi_{n-1}(\om),\ \ n\ge 1,
\ee
where $\M_{n-1}(\om)$ is an $m\times m$ matrix-function, whose matrix elements  are trigonometric polynomials  given by
\be \label{matr}
(\M_{n-1}(\om))(b,c) = \sum_{j \le k_b:\ u_j^{(b)} = c} \exp\left[-2\pi i \om \left(|\zeta^{n-1}(u_1^{(b)})| + \cdots +|\zeta^{n-1}(u_{j-1}^{(b)})|\right)\right] 
\ee
Note that $\M_n(0) =\Sf^{\sf T}= \Sf_\zeta^{\sf T}$, the transpose of the substitution matrix, for all $n\in \N$. 
Since $\vec{\Psi}_0^{(a)}(\om)= \bbe_a$ (the basis vector corresponding to $a\in \A$), it follows from (\ref{eq-matr1}) that
\be \label{eq-matr2}
\vec{\Psi}_{n}^{(a)}(\om) = \M_{n-1}(\om) \M_{n-2}(\om) \cdots \M_0(\om) \bbe_a,
\ee
hence
\be \label{def-Pbi}
\Pbi_n(\om)= \M_{n-1} (\om)\M_{n-2}(\om)\cdots \M_0(\om).
\ee 
Denote by $\Sigb_\zeta$ the $m\times m$ matrix of correlation measures for $\zeta$; that is, $\Sigb_\zeta(a,b) = \sig_{a,b}$. Below $\one$ denotes the vector $[1,\ldots,1]^t$ of all 1's and
$\langle\vec{x},\vec{y}\rangle$ stands for the scalar product in $\R^m$. We write $x_n\sim y_n$ when $\lim_{n\to \infty} x_n/y_n=1$. The following is a generalization of
\cite[Theorem 8.1]{Queff} by Queffelec to the non-constant length case.

\begin{lemma} \label{lem-Riesz}
Let $\theta$ be the Perron-Frobenius eigenvalue of the substitution matrix $\Sf$ and let $\vec{r},\vec{\ell}$ be respectively the (right) eigenvectors of $\Sf,\Sf^t$ corresponding to $\theta$, normalized by the condition $\langle\vec{r},\vec{\ell}\rangle =1$. Then
\be \label{eq-Sigb}
 \Sigb_\zeta = \frac{1}{\langle\vec{r},\one\rangle \langle\one,\vec{\ell}\rangle}\cdot\mbox{\rm weak*-}\lim_{n\to \infty} \theta^{-n} \ov{\Pbi_n^*(\om) \,\Pbi_n(\om)}\,d\om.
\ee
\end{lemma}

\begin{proof}
By (\ref{def-Psi}), the $(a,b)$ entry of $\ov{\Pbi_n^*(\om) \Pbi_n(\om)}$ is $\sum_{j=1}^{m} {\Phi_a(\zeta^n(j),\om)}\cdot\ov{\Phi_b(\zeta^n(j),\om)}$.
We have
$$
|\zeta^n(j)| = \langle\Sf^n\bbe_j, \one\rangle \sim \theta^n \langle \bbe_j, \vec{\ell}\rangle  \langle \vec{r}, \one\rangle
$$
by the Perron-Frobenius Theorem, hence
$$
\theta^{-n} \sum_{j=1}^{m} {\Phi_a(\zeta^n(j),\om)}\cdot\ov{\Phi_b(\zeta^n(j),\om)} \sim \sum_{j=1}^{m} \frac{\langle \bbe_j, \vec{\ell}\rangle \langle  \vec{r}, \one\rangle}{|\zeta^n(j)|} \ {\Phi_a(\zeta^n(j),\om)}\cdot\ov{\Phi_b(\zeta^n(j),\om)},
$$
and the desired claim follows from (\ref{corr2}), in view of the fact that  $\one = \sum_{j=1}^{m} \bbe_j$.
\end{proof}


\subsection{Other spectral measures and maximal spectral type} \label{sec:other}

Although it may seem that the cylindrical functions are very special,  they are fairly representative. In fact, any $L^2$ function on $X_\zeta$ can be approximated by functions that depend
on finitely many coordinates, and so it suffices to understand the spectral measures of such in order to obtain the maximal spectral type.
To make this precise, we briefly describe a general construction.
The substitution $\zeta$ can  be extended to $\Ak^\Z$, so that $\zeta(x_0)$ starts from 0-th position. Consider  the sets $\zeta^k[a]$, $a\in \Ak$, $k\ge 0$. It is proved in \cite[5.6.3]{Queff} that the partitions
$$
\Pk_k = \{T^i(\zeta^k[a]),\ a\in \Ak, \ 0 \le i < |\zeta^k(a)|\},\ \ k\ge 0,
$$
generate the Borel $\sig$-algebra on $X_\zeta$.
It follows that the maximal spectral type of $(X_\zeta,T,\mu)$ is equivalent to 
\be \label{eq:maxispec}
\sum_{k\ge 0,\ a\in \Ak} 2^{-k} \sig_{\Onne_{\zeta^k[a]}}.
\ee
The spectral measures $\sig_{\Onne_{\zeta^k[a]}}$ can be analyzed similarly to the correlation measures $\sig_a$. Namely, one can show that the matrix of measures
$
\Sigb_\zeta^{(k)}:= [\sig_{\Onne_{\zeta^k[a]},\Onne_{\zeta^k[b]}}]_{a,b\in \Ak}
$
can be expressed, analogously to (\ref{eq-Sigb}), as
$$
 \Sigb^{(k)}_\zeta = \frac{1}{\langle \vec{r},\one\rangle \langle \one,\vec{\ell}\rangle }\cdot\mbox{\rm weak*-}\lim_{n\to \infty} \theta^{-n} \ov{\bigl(\Pbi_n^{(k)}\bigr)^*(\om) \,\Pbi^{(k)}_n(\om)} \,d\om,
$$
where
$
\Pbi^{(k)}_n(\om) = \M_{n+k-1}(\om)\cdots \M_{k}(\om).
$


\subsection{Singularity of some Riesz products and spectral measures}

Here we will show that the  Thue-Morse substitution automorphism has pure singular spectrum, following the approach of Kakutani \cite{Kaku} { (somewhat implicitly, this was already proved by K. Mahler in 1927 \cite{Mah27})}. In fact, we will do it in a more general case, for
an arbitrary bijective substitution on two symbols. The proof is a modification of an argument of
Baake, G\"ahler, and Grimm \cite{BGG}; see also \cite[10.1]{BG1}.

\begin{prop} \label{prop:singular}
Let $\zeta$ be a primitive bijective substitution on $\{0,1\}$ of length $q$. More precisely, let $\zeta(0) = u_0\ldots u_{q-1}$ and $\zeta(1) = \ov{u}_0\ldots \ov{u}_{q-1}$ where
$\ov{0}=1,\ \ov{1} = 0$. Then the associated substitution automorphism has purely singular spectrum.
\end{prop}

\begin{proof}[Proof sketch] We will prove that the spectral measure $\sig_f$, where $f(x) = (-1)^{x_0}$, is purely singular. As already mentioned, the functions that are even with respect to the
involution $a\mapsto \ov a$ lie in the closed linear span of the eigenfunctions, and the remaining spectral measures generating the measure
 in \eqref{eq:maxispec} may be analyzed exactly as $\sig_f$. { We obtain
\be \label{eq-Riesz0}
\sig_f = \mbox{\rm weak$^*$-}\lim_{n\to \infty} \frac{1}{q^n} \prod_{j=0}^{n-1} \left|F(q^j \om)\right|^2 d\om,\ \ \ \mbox{where}\ \ F(\om) = \sum_{k=0}^{q-1} (-1)^{u_k} e^{-2\pi i k\om},
\ee
where the weak$^*$ convergence can be shown as in \eqref{corr222} and \eqref{eq-TM1}. In fact, weak convergence holds more generally, see \cite[Section 1.3]{Queff}.}
Thus $\sig_f$ has the form \eqref{eq-Riesz} below, with $P(\om) = \frac{1}{q}|F(\om)|^2$, which is a non-constant trigonometric polynomial with a constant term 1. Now singularity of $\sig_f$ follows from
the next lemma.
\end{proof}

\begin{lemma}
Let $q\ge 2$ be an integer and
$$
P(\om) = 1 + \sum_{k = -q+1,\ k\ne 0}^{q-1} a_k e^{-2\pi i k\om},\ \ a_k \in \C,
$$
a non-constant non-negative trigonometric polynomial. Suppose that 
\be \label{eq-Riesz}
\nu = \mbox{\rm weak$^*$-}\lim_{n\to \infty} \prod_{j=0}^{n-1} P(q^j \om) \,d\om,
\ee
where we assume that the weak$^*$ limit of the absolutely continuous measures in the right-hand side exists, and is a positive finite measure on $[0,1]$. Then $\nu$ is purely
singular.
\end{lemma}

\begin{proof}
Observe that \eqref{eq-Riesz} defines a 1-periodic measure on $\R$, since the finite products in the right-hand side are 1-periodic for all $n$. 
Denote by $\nu(\om/q)$ the push-forward measure of $\nu$ under the map
$\om\mapsto q \om$, defined on $\R$. Then we have
\begin{eqnarray}
\nu(\om/q) & = & \mbox{weak$^*$-}\lim_{n\to \infty} \prod_{j=0}^{n-1} P(q^{j-1} \om) \,d(\om/q) \nonumber \\[-2ex]
                 & = & (1/q) P(\om/q) \cdot \mbox{weak$^*$-}\lim_{n\to \infty} \prod_{j=0}^{n-2} P(q^j \om) \,d\om \nonumber \\
                 & = & (1/q) P(\om/q) \cdot \nu(\om). \label{eq-Riesz2}
\end{eqnarray}
Now, following the approach of Kakutani \cite{Kaku}, we can take the absolutely continuous 
 part of both sides of the equation \eqref{eq-Riesz2}, since the push-forward under $\om\mapsto q\om$ and
multiplication by a polynomial, that does not vanish on a set of positive measure, commute with the Lebesgue decomposition. Thus, denoting $\nu_{\rm ac}$ by $\eta$,
\be \label{eq-ac1}
\eta(\om/q) = (1/q) P(\om/q) \cdot\eta(\om).
\ee
The goal is to show that a non-zero absolutely continuous measure cannot satisfy \eqref{eq-ac1}.

First we observe that $\eta$ from \eqref{eq-ac1} must be invariant under the ``times $q$'' map $S_q(\om) = q\om$ (mod 1) on $[0,1]$. Indeed, this invariance is equivalent to 
\be \label{eq-inva}
\eta = \eta\circ S_q^{-1} = \eta\left(\frac{\om}{q}\right) + \eta\left(\frac{\om+1}{q}\right) + \cdots + \eta\left(\frac{\om+ (q-1)}{q}\right).
\ee
Since $\eta$ is 1-periodic, we have from \eqref{eq-ac1} that $\eta(\frac{\om+j}{q}) = \frac{1}{q} P(\frac{\om+j}{q})\eta$. Thus the claimed invariance is equivalent to
$
\sum_{j=0}^{q-1} P\left(\frac{\om+j}{q}\right) = q,
$
which is immediate from the definition of $P$. Another characterization of $S_q$-invariance of $\eta$ is $\what{\eta}(n) = \what{\eta}(qn)$ for $n\in \Z$, which follows from the definition of
the Fourier coefficients. Thus the Lebesgue measure on $[0,1]$ is the only $S_q$-invariant absolutely continuous measure, by the Riemann-Lebesgue Lemma. It remains to observe that the Lebesgue
measure does not satisfy \eqref{eq-ac1}, if $P$ is non-constant.
\end{proof}

We should note that there are several alternative proofs of more general statements. For more on the rich theory of 
generalized Riesz products and how they appear in the spectral theory of dynamical systems, see the book  by Nadkarni \cite[Ch.\,15]{Nad}, the papers by el Abdalaoui and Nadkarni \cite{AbNad},
el Abdalaoui \cite{Ab2}, and references therein. {Weak convergence in \eqref{eq-Riesz} is known, for instance, when $P(\om) = \frac{1}{q} {\Bigl| \sum_{k=0}^{q-1} z_k e^{-2\pi i k\om}\Bigr|}^2$,
with $z_k$ complex, $|z_k|\le 1$, see \cite[Section 1.3]{Queff}.}


\section{Suspension flows and the spectral cocycle}
In order to use the full power of the Riesz matrix product representation of spectral measures, it is useful to 
pass from substitution automorphisms (i.e., $\Z$-actions) to flows (i.e., $\R$-actions).
Let $\zeta$ be a primitive substitution on $\Ak=\{1,\ldots,d\}$, and let $(X_\zeta,T)$ be the corresponding uniquely ergodic $\Z$-action. For a strictly positive vector $\vec{s} = (s_1,\ldots,s_d)$ we consider the {\em suspension flow} over $T$, with the piecewise-constant roof function, equal to $s_j$ on the cylinder set $[j]$. It is defined as follows. Let
$$
\Xxi^{\vec{s}} = \{y=(x,t):\ x\in X_\zeta,\ 0 \le t \le s_{x_0}\}/_\sim\ ,
$$
 where  the relation $\sim$ identifies the points $(x, s_{x_0})$ and $(T (x), 0)$. The measure $\wt{\mu}$ on $\Xxi^{\vec{s}}$ is induced from the product  of 
 $\mu$ on the cylinder sets $[a]\subset X_\zeta$ and the Lebesgue measure; it is invariant under the action $h_\tau(x,t) = (x, t+\tau)$, which is defined for all $\tau\in \R$ using the identification 
 $\sim$.
The resulting space will be denoted by $\Xxi^{\vec{s}}$ and the flow
by $(\Xxi^{\vec{s}},h_t,\wt\mu)$. 

An additional motivation comes from the fact that
these flows can also be viewed as {\em tiling dynamical systems}, with the ``prototiles'' that are closed line segments of length $s_j$ for $j=1,\ldots,d$.
 A tiling of $\R$ is obtained from a pair $(x,t): x\in X_\zeta,\ 0 \le t\le s_{x_0},$ by 
considering a bi-infinite string of line segments labelled $x_j$ of length $s_{x_j}$, $j\in \Z$,
in such a way that the origin is contained in the tile corresponding to $x_0$, at the distance of $t$ from the left
endpoint. The group $\R$ acts on this space by (left) translation, and it is easy to see that this definition is equivalent to the one with the suspension flow.
Such dynamical systems have been studied, e.g., in \cite{BeRa,CS03,SolTil}.
A special role is played by the {\em self-similar suspension flow}, or tiling dynamical
system, corresponding to the choice of $\vec{s}$ as the PF eigenvector of the transpose substitution matrix $\Sf_\zeta^{\sf T}$. Then one can define a {\em geometric} substitution/renormalization action $Z$ on $\Xxi^{\vec{s}}$: dilate
 the original tiling by a factor of the PF eigenvalue $\theta$ and then subdivide the resulting tiles according to the substitution rule. These two actions satisfy the commutation relation
$
Z\circ h_\tau = h_{\theta \tau} \circ Z.
$

Let $f\in L^2(\Xxi^{\vec{s}},\wtil{\mu})$. By the Spectral Theorem for measure-preserving flows, there is a finite positive Borel measure $\sig_f$ on $\R$ such that
$$
\int_{-\infty}^\infty e^{2 \pi i\om \tau}\,d\sig_f(\om) = \langle f\circ h_\tau, f\rangle\ \ \ \mbox{for}\ \tau\in \R.
$$


In general, there is no simple relation between the spectrum of a $\Z$-action on the circle and the spectrum of the suspension flow, except in the special case where the roof function is constant.
(Note that the constant roof function suspension is self-similar if and only if the substitution is of constant length.) The following lemma was proved in \cite{BerSol}; it is a simple
consequence of \cite{Goldberg}.

\begin{lemma}    \label{lem-equivalence}
Let $(X,T,\mu)$ be an invertible probability-preserving system, and let $(\wt{X}, h_\tau, \wt{\mu})$ be the suspension flow with the constant-one roof function.
For  $f\in L^2(X,\mu)$ consider the
function $F\in L^2(\wt{X},\wt{\mu})$  defined by
$F(x,t)=f(x)$. Then
the following relation holds between the spectral measures $\sigma_F$  on $\R$ and $\sigma_f$ on $\T$:
\[
d\sig_F(\om)=\left(\frac{\sin(\pi \om)}{\pi \om}\right)^2 \cdot d\sig_f(e^{2\pi i \om}),\ \ \om\in \R.
\]
\end{lemma}


The next construction extends the considerations of Section~\ref{sec:Riesz} to suspension flows. 

\begin{defi} Let $\vec s\in \R^d_+$.
For a word $v$ in the alphabet $\Ak$ denote by $\vec{\ell}(v)\in \Z^d$ its ``population vector'' whose $j$-th entry is the number of $j$'s in $v$, for $j\le d$. The
``tiling length'' of $v$ is defined by 
\be \label{tilength}
|v|_{\vec{s}}:= \langle\vec{\ell}(v), \vec{s}\rangle.
\ee
\end{defi}

The analog of the exponential sum in \eqref{def-Phi}, with $\varphi = \delta_a$, is
\be \label{def-Phi3}
\Phi_a^{\vec{s}}(v,\om) = \sum_{j=0}^{|v|-1} \delta_a(v_j) \exp(-2\pi i \om |v_0\ldots v_j|_{\vec{s}}).
\ee
{\em Lip-cylindrical functions} for a suspension flow are defined by
\be \label{eq-cylinder}
f(x,t) = \sum_{j\in \A} \Onne_{[j]} \cdot \psi_j(t),\ \ \ \mbox{where}\ \ \psi_j \ \mbox{is Lipschitz on}\ [0,s_j].
\ee
 A generalization of \eqref{corr2} and \eqref{corr222} is 
\be \label{spect}
\sig_{f} = \mbox{\rm weak*-}\lim_{N\to \infty} \frac{1}{|\zeta^N(\gam)|_{\vec{s}}}\,\sum_{a,b\in \A} \what\psi_a(\om)\ov{\what{\psi}_b(\om)}\cdot\Phi^{\vec{s}}_a(\zeta^N(\gam),\om)\,\ov{\Phi^{\vec{s}}_b(\zeta^N(\gam),\om)}\, d\om,
\ee
where the right-hand side represents the weak$^*$ limit of a sequence of absolutely continuous measures on $\R$. 


\subsection{Digression: quasicrystals and diffraction spectrum} \label{sec:diffract}

In the mathematical physics of quasicrystals and the theory of ``aperiodic order'', one of the basic notions is {\em diffraction spectrum}.
We  sketch the definition briefly, following the approach of Hof \cite{Hof95,Hof97}; see also \cite[Chapter 9]{BaGr}. 

Let $\Lam$ be a {\em Delone set} in $\R^n$, that is, a uniformly discrete, relatively dense set (for the systems treated in this paper, $n=1$, but the definition is the same in any dimension).
 The so-called ``Dirac comb'' measure $\delta_\Lam:= \sum_{x\in \Lam}\delta_x$
is a simple model of an infinite configuration of atoms, which in turn is a model of a solid. Roughly speaking, the diffraction spectrum is described by the 
{\em diffraction measure}, which is the Fourier transform of an autocorrelation. More precisely, let $B_R$ be the ball of radius $R$ in $\R^n$. Consider the volume 
averaged convolution
$$
\gam^{(R)}_\Lam:= \frac{\delta_\Lam|_{B_R} * \wt\delta_\Lam|_{B_R}}{{\rm vol} (B_R)},
$$
where $\wt \nu(x) = \ov{{\nu}(-x)}$.
Every weak$^*$ accumulation point of $\gam^{(R)}_\Lam$, as $R\to \infty$, is called an autocorrelation measure. By construction, it is a positive definite measure, hence its
Fourier transform is a positive measure. Under some natural assumptions (i.e., when the point set $\Lam$ arises from a primitive substitution), there is a unique autocorrelation
$\gam_\Lam = \mbox{\rm weak*-}\lim_{R\to \infty} \gam_\Lam^{(R)}$. The {\em diffraction spectrum measure} is, by definition,
$$
\Sig_{\rm diffr}(\Lam) = \what{\gam_\Lam}.
$$
This diffraction spectrum is closely related to the dynamical spectrum considered in this paper. In particular, it easily follows from \eqref{spect} that
$$
\sig_{f_a} =  |\what{\psi}_a|^2\cdot \Sig_{\rm diffr}(\Lam_a),\ \ \ \mbox{where}\ \ f_a = \Onne_{[a]}\cdot \psi_a,
$$
and $\Sig_{\rm diffr}(\Lam_a)$ is the diffraction spectrum measure, corresponding to 
 the set $\Lam_a\subset \R$ of the left endpoints of $a$-tiles in the 1-dimensional substitution tiling. Such a connection (also for more general tilings and point systems in $\R^n$)
 was first established
by Dworkin \cite{Dworkin}, see also \cite{Hof95,Hof97}. 

\subsection{Definition of the spectral cocycle}

The formula \eqref{spect} and the matrix Riesz products of Section 3 suggest the following  {\em cocycle} construction, introduced in \cite{BuSo20}.

\smallskip

Write $\bz = (z_1,\ldots,z_{d})$ and $\bz^v = z_{v_0}z_{v_1}\ldots z_{v_k}$ for a word $v = v_0v_1\ldots v_k\in \A^k$.
As above, suppose that
$\zeta(b) = u_1^{b}\ldots u_{k_b}^{b},\ \ b\in \Ak.$
Define a matrix-function on $\T^d$ whose entries are polynomials in $\bz$-variables:
$$
\Mc_\zeta(\bz) = [\Mc_\zeta(z_1,\ldots,z_{d})]_{b,c} = \Bigl( \sum_{j\le k_b,\ u_j^{b} = c} \bz^{u^b_1\ldots u^b_{j-1}} \Bigr)_{(b,c)\in \A^2},\ \ \ \bz\in \T^d,
$$
where $j=1$ corresponds to $\bz^{\es}=1$. 

Whereas the $\bz$-notation is helpful, we will actually need to lift $\Mc_\zeta$ to the universal cover, in other words, write
$\bz = \exp(-2 \pi i \bxi)$, where $\xi = (\xi_1,\ldots,\xi_{d})$. Thus we obtain a $\Z^d$-periodic matrix-valued function
 function, which we denote by the same letter.
 
 \begin{defi} \label{def-matrix} Let
  $\Cc_\zeta:  \R^d\to M_d(\C)$  (the space of complex $d\times d$ matrices): 
 \be \label{coc0}
\Cc_\zeta(\xi) = [\Cc_\zeta(\xi_1\ldots,\xi_d)]_{(b,c)} := \Bigl( \sum_{j\le k_b,\ u_j^{b} = c} \exp\bigl(-2\pi i \sum_{k=1}^{j-1} \xi_{u_k^{b}}\bigr)\Bigr)_{(b,c)\in \A^2},\ \ \ \xi\in \R^d.
\ee
\end{defi}

\noindent {\bf Example.}
Consider the substitution $\zeta: 1\mapsto 1112,\ 2\mapsto 123,\ 3 \mapsto 2$. Then
$$
\Mc_{\zeta}(z_1,z_2,z_3) = \left( \begin{array}{ccc} 1 + z_1 + z_1^2 & z_1^3 & 0 \\ 1 & z_1 & z_1 z_2 \\ 0 & 1 & 0 \end{array} \right), \ \ \ z_j = e^{2\pi i \xi_j}.
$$

\medskip

Observe that
$\Mc_\zeta(0) = \Sf_\zeta^{\sf T}$;
the entries of the matrix $\Mc_\zeta(\xi)$ are trigonometric polynomials with coefficients 0's and 1's. The entries of $\Mc_\zeta(\xi)$ are less than or equal to 
 the corresponding entries of $\Sf_\zeta^{\sf T}$
 in absolute value for every $\bz\in \T^d$.
Crucially, the following
{\em cocycle condition} holds:
 for any substitutions $\zeta_1,\zeta_2$ on $\Ak$,
\be \label{most}
\Cc_{\zeta_1\circ \zeta_2}(\xi) = \Cc_{\zeta_2}(\Sf^{\sf T}_{\zeta_1}\xi)\Cc_{\zeta_1}(\xi),
\ee
which is verified by a direct computation.

\begin{defi}
Suppose that $\det(\Sf_\zeta)\ne 0$ and consider the  endomorphism of the torus $\T^d$
\begin{equation}\label{torend}
\xi \mapsto \Sf_\zeta^{\sf T} \xi \  (\mathrm{mod} \   \Z^d),
\end{equation} 
 which preserves the Haar measure $m_d$. Then
 \be \label{cocycle3}
\Cc_\zeta(\xi,n):= \Cc_\zeta\bigl((\Sf_\zeta^{\sf T})^{n-1}\xi \bigr)\cdot \ldots \cdot \Cc_\zeta(\xi),
\ee
  is called the {\em spectral cocycle}, associated to $\zeta$, over the endomorphism \eqref{torend}.
 \end{defi}
 
  Note that (\ref{most}) implies
\be \label{coc2}
\Mc_\zeta(\xi,n) =\Cc_{\zeta^n}(\xi),\ \ n\in \N,
\ee
and the matrix defined in \eqref{matr} satisfies $\M_n(\om) = \Cc_\zeta(\om\one,n)$. It will be useful to rewrite the formula for $\Cc_{\zeta^n}(\xi)$ differently, emphasizing the connection with \eqref{def-Phi3} and \eqref{matr}. For $\xi = \om\vec s$, we can represent the matrix elements of $\Cc_\zeta(\om\vec s)$ as
\be \label{cu1}
\bigl[\Cc_\zeta(\om\vec s)\bigr]_{(b,c)} = \sum_{j\le k_b, \, u_j^b = c} \exp\bigl(-2\pi i \om \bigl|p^b_j\bigr|_{\vec s}\bigr)
= \sum_{j\le k_b, \, u_j^b = c} \exp\bigl(-2\pi i \om \bigl\langle \vec{\ell}(p^b_j), {\vec s}\bigr\rangle\bigr),
\ee
where $p^b_j$ is the prefix of the $j$-th symbol (which is $c$) in the word $\zeta(b)$. By the definition of the substitution matrix $\Sf_\zeta$, for any letter and hence for any word 
$v\in \A^+$, holds $\vec\ell(\zeta(v)) = \Sf_\zeta \vec\ell(v)$, so
\be \label{Der}
\bigl\langle \vec{\ell}(v), (\Sf_\zeta^{\sf T})^k {\vec s}\bigr\rangle  = \bigl\langle \Sf_\zeta^k \vec{\ell}(v), {\vec s}\bigr\rangle = 
\bigl\langle \vec{\ell}(\zeta^k(v)), \vec s\bigr\rangle = \bigl| \zeta^k(v) \bigr|_{\vec s}.
\ee
Thus,
\be \label{cu2}
\bigl[\Cc_\zeta\bigl( (\Sf_\zeta^{\sf T})^k \om\vec s\bigr)\bigr]_{(b,c)} = \sum_{j\le k_b, \, u_j^b = c} \exp\bigl(-2\pi i \om \bigl|\zeta^k(p^b_j)\bigr|_{\vec s}\bigr\rangle\bigr).
\ee

\begin{remark}
{\em
If $\det(\Sf_\zeta)=0$, one can restrict $\Sf_\zeta^{\sf T}$ to a rational invariant subspace on which it is non-singular and consider the spectral cocycle on the corresponding invariant
sub-torus, see \cite[Remark 4.6]{Yaari}.
}
\end{remark}


\subsection{Lyapunov exponents and the dimension of spectral measures}

Consider the pointwise upper  Lyapunov exponent of our cocycle, corresponding to a given vector $\vec{z}\in \C^m$: 
\be \label{Lyap1}
{\chi}_{\zeta,\xi,\vec{z}}^+=\limsup_{n\to \infty} \frac{1}{n} \log \|\Cc_\zeta(\xi,n)\vec{z}\|, 
\ee
as well as
\be \label{Lyap2}
{\chi}_{\zeta,\xi}^+= \limsup_{n\to \infty} \frac{1}{n} \log \|\Cc_\zeta(\xi,n)\|.
\ee
The definition is independent of the  norm  used. Of course, ${\chi}_{\zeta,\xi,\vec{z}}^+ \le {\chi}_{\zeta,\xi}^+$ for all $\vec z$.
Since  for every $\xi$ the absolute values of the entries of $\Cc_\zeta(\xi,n)$ are not greater than those of $\Sf_{\zeta^n}^{\rm T}$,
$$\chi_{\zeta,\xi}^+ \le \log\theta\ \ \mbox{ for all}\ \xi\in \T^d,$$
where $\theta$ is the PF eigenvalue of $\Sf_\zeta$. The pointwise upper Lyapunov exponent $\chi_{\zeta,\xi}^+$ is invariant under the action of the endomorphism induced by
$\Sf_{\zeta}^{\rm T}$. 
 If this endomorphism is ergodic,
which is equivalent to $\Sf_\zeta$ having no eigenvalues that are roots of unity (see \cite[Cor,\,2.20]{EW}),
then the limit in \eqref{Lyap2} exists and is constant $m_d$-a.e. by the Furstenberg-Kesten Theorem \cite{FK}; it will be denoted
by $\chi_\zeta$. This is the {\em top Lyapunov exponent} of the spectral cocycle (other exponents exist a.e.\ as well, by the Oseledets Theorem).

We are interested in fractal properties of spectral measures. One of the commonly used local characteristics of a finite  measure $\nu$ on a metric space is the {\em lower local dimension} defined 
by $$\underline{d}(\nu,\om) = \liminf_{r\to 0} \frac{\log\nu(B_r(\om))}{\log r}\,.$$ Alternatively, it can be thought of as the best possible local H\"older exponent, i.e.:
$$
\und{d}(\nu,\om) = \sup\{\alpha\ge 0:\ \nu(B_r(\om)) = O(r^\alpha),\ r\to 0\}.
$$
For example, the lower local dimension is zero at a point mass and is infinite outside the compact support of a measure.
The following is a special case of \cite[Theorem 4.3]{BuSo20}. 

\begin{theorem} \label{thm:local1}
Let $\zeta$ be a primitive aperiodic substitution on $\A = \{1,\ldots,d\}$ with a non-singular substitution matrix $\Sf_\zeta$.
Let $\vec s\in \R^d_+$ and consider the suspension flow $(\Xx_\zeta^{\vec{s}}, h_t, \wt \mu)$.
Let $f(x,t) = \sum_{j\in \A} \Onne_{[j]} \cdot \psi_j(t)$ be a Lip-cylindrical function and $\sig_f$ the corresponding spectral measure.
 Fix $\om\in \R$ and let 
 $$
\xi = \om\vec{s} \ \mbox{\rm (mod $\Z^m$)},  \ \vec{z} = (\what{\psi}_j(\om))_{j\in \Ak}.
$$
Suppose that  ${\chi}^+_{\zeta,\xi,\vec z}> 0$.
Then 
 \be \label{ki2}
 \underline{d}(\sig_{f},\om) = 2 - \frac{2{\chi}^+_{\zeta,\xi,\vec z}}{\log\theta}\,.
 \ee
 If ${\chi}^+_{\zeta,\xi,\vec z}\le 0$, then
\be \label{kia}
\underline{d}(\sig_f,\om) \ge 2.
\ee
\end{theorem}

The starting point of the proof is the following ``spectral lemma,'' inspired by Hof \cite{Hof}. The 1st part is the analog of Lemma~\ref{lem-spec1} for measure-preserving flows.

\begin{lemma} \label{lem-spec2}
For $f \in L^2(\Xx_\zeta^{\vec s},\wt\mu)$, $R>0$,  $\om \in \R$, let
$$
G_R(f,\om) = R^{-1} \int_{\Xx_\zeta^{\vec s}} |S_R^{(x,t)}(f,\om)|^2\,d\wt\mu(x,t),\  \ \ \mbox{where}\ \  \ S_R^{(x,t)} (f,\om) = \int_0^R e^{-2\pi i \om \tau} f(h_\tau (x,t))\,d\tau.
$$
Let $\sig_f$ be the spectral measure for the flow.

{\rm (i)} Suppose that for some fixed $\om \in \R$, $\alpha \ge 0$, { fixed} $R\ge 1$, and $f\in L^2(\Xx_\zeta^{\vec s},\wt\mu)$,
$$
G_R(f,\om) \le C_1 R^{1-\alpha}.
$$
Then
$$
\sig_f(B_r(\om)) \le C_1 \pi^2 2^\alpha  r^\alpha,\ \ \mbox{for}\ r = (2R)^{-1}.
$$

{\rm (ii)} Suppose that for some fixed $\om \in \R$, $\alpha\in (0,2)$, $r_0\in (0,1)$, and $f\in L^2(\Xx_\zeta^{\vec s},\wt\mu)$,
$$
\sig_f(B_r(\om)) \le C_2 r^\alpha,\ \ \mbox{for all}\ \ r\in (0,r_0).
$$
Then
$$G_R(f,\om) \le C_3 R^{1-\alpha},\ \ \mbox{for all}\ \ R\ge r_0^{-\frac{2}{2-\alpha}},$$
where $C_3$ depends only on $\alpha$, $C_2$, and $f$.
\end{lemma}

A rough scheme of the proof of Theorem~\ref{thm:local1} is as follows.
 For the estimate from below on the lower local dimension $\und{d}(\sig,\om)$ we obtain upper bounds
 of the Birkhoff integrals $S_R^{(x,t)}$ that are {\em uniform} in $(x,t)$, and then use part (i) of Lemma~\ref{lem-spec2}. 
For $x$ starting with $\zeta^n(\gam)$ for some $\gam\in \A$ the Birkhoff integral looks almost like the expression in 
\eqref{def-Phi3}, and for an arbitrary $x$ one can use an efficient decomposition of $x[1,N]$ into such substituted symbols.

For the estimate from above on the lower local dimension $\und{d}(\sig,\om)$ we use part (ii) of Lemma~\ref{lem-spec2} to obtain lower bounds on $G_R(f,\om)$ on a sequence of scales $R=R_n\to \infty$ corresponding to the tiling lengths
of the words $\zeta^n(\gam)$ and restricting the integration in the definition of $G_R(f,\om)$ to the portion of the space where $x$ ``almost'' starts with $\zeta^n(\gam)$. This part is more delicate, see  \cite{BuSo20} for details. \qed

\medskip

We point out two corollaries from \cite{BuSo20}, where we make use of {\em simple cylindrical functions}:
\be \label{simple}
f_{\vec b}(x,t)=\sum_{j\in \Ak} b_j \Onne_{[j]}(x),\ \ \vec b\in \R^d_+.
\ee

\begin{corollary} \label{cor:local1}
Suppose we are under the assumptions of Theorem~\ref{thm:local1} and $\vec s\in \R^d_+$ is fixed. Then

{\rm (i)} $\chi^+_{\zeta,\om\vec s} \le \half \log\theta$ for Lebesgue-a.e.\ $\om\in \R$.

{\rm (ii)} $\chi^+_{\zeta,\om\vec s} \ge \half \log\theta$ for $\sig_f$-a.e.\ $\om\in\R$, for a.e.\ simple cylindrical function $f$.

{\rm (iii)} If $\chi^+_{\zeta,\om\vec s} < \half \log\theta$ for Lebesgue-a.e.\ $\om\in \R$, then the flow $(\Xx_\zeta^{\vec{s}}, h_t, \wt \mu)$ has purely singular spectrum.
\end{corollary}

\begin{corollary} \label{cor:local2}
Let $\zeta$ be a primitive aperiodic substitution on $\A = \{1,\ldots,d\}$ with a non-singular substitution matrix $\Sf_\zeta$ with no eigenvalues that are roots of unity, so that the Lyapunov exponent
$\chi_\zeta$ is well-defined. Then

{\rm (i)} the 
Lyapunov exponent satisfies $\chi_{\zeta} \le \half \log\theta$;

{\rm (ii)} if $\chi_{\zeta} < \half \log\theta$, then for a.e.\ $\vec s\in \R^d_+$ the flow $(\Xx_\zeta^{\vec{s}}, h_t, \wt \mu)$ has purely singular spectrum.
\end{corollary}

 Corollary~\ref{cor:local2} follows from Corollary~\ref{cor:local1} and the existence $m_d$-a.e.\ of the global Lyapunov exponent.

\begin{proof}[Sketch of the proof of Corollary~\ref{cor:local1}]
(i)
An application of Fubini's Theorem, combined with \eqref{ki2}, yields that for Lebesgue-a.e.\ $\vec b\in \R^d_+$, for Lebesgue-a.e.\ $\om\in \R$, 
\be \label{kiaa}
\underline{d}(\sig_{f_{\vec b}}) = 2 - \frac{2\chi^+_{\zeta,\om \vec s}}{\log \theta},
\ee
where $f_{\vec b}$ is a simple cylindrical function from \eqref{simple}, see \cite[Cor.\,4.4]{BuSo20}.
Assuming $\chi^+_{\zeta,\om\vec s}>\half\log\theta$ for $\om\in A\subset \R$, where $A$ has positive Lebesgue measure,  we obtain from \eqref{kiaa}
that there is a non-zero simple
cylindrical function $f_{\vec b}$ such that $\underline{d}(\sig_{f_{\vec b}}) < 1$ on $A$. But this contradicts the well-known fact that $\lim_{r\to 0} \frac{\nu(B_r(\om))}{2r} <
\infty$ Lebesgue-a.e.\ for any positive measure $\nu$ on the line.

\smallskip

(ii) Let $f= f_{\vec b}$ be a simple cylindrical function for which \eqref{kiaa} holds. Observe that $\sig_f$, being a measure on the line, has Hausdorff dimension not greater 
than one, hence
$$
\dim_H(\sig_f)  = \sup\{s\ge 0:\ \underline{d}(\sig_f,\om) \ge s\ \ \mbox{for $\sig_f$-a.e.\ $\om$}\} \le 1,
$$
see \cite[Prop.\,10.2]{Falc-book}.
It follows that $\underline{d}(\sig_f,\om)\le 1$ for $\sig_f$-a.e.\ $\om$, and hence $\chi^+_{\zeta,\om\vec s}\ge \half\log\theta$ for $\sig_f$-a.e.\ $\om$ by \eqref{kiaa}.

\smallskip

(iii) Suppose that $\chi^+_{\zeta,\om\vec s} < \half\log\theta$ for  Lebesgue-a.e.\ $\om\in \R$.
We will show that $\sig_f$ is singular for every Lip-cylindrical function $f$. By Theorem~\ref{thm:local1},
$$
\underline{d}(\sig_f,\om) \ge 2 - \frac{2\max\{0,\chi^+_{\zeta,\om\vec s}\}}{\log\theta} >1
$$
for  Lebesgue-a.e.\ $\om\in \R$. But this implies $\lim_{r\to 0} \frac{\sig_f(B_r(\om))}{2r}=0$ for Lebesgue-a.e.\ $\om\in \R$, hence the absolutely continuous part of
$\sig_f$ is zero. To conclude, we need to apply a similar reasoning to higher-level cylindrical functions, like those appearing in \eqref{eq:maxispec}.
\end{proof}

It is not immediately clear how to apply the criteria for singularity from Corollary~\ref{cor:local1}(iii) and  Corollary~\ref{cor:local2}(ii) to specific systems, but recently
we obtained the following

\begin{theorem}[{\cite{BuSo20b}}] \label{thm1}
Let $\zeta$ be a primitive aperiodic substitution on $\Ak = \{0,\ldots,d-1\}$, with $d\ge 2$, such that  the substitution matrix $\Sf_\zeta$ has a characteristic polynomial irreducible over $\Q$. Let $\theta$ be the Perron-Frobenius eigenvalue of $\Sf_\zeta$. If
\be \label{cond1}
\chi_\zeta < \frac{\log\theta}{2},
\ee
then the substitution $\Z$-action has pure singular spectrum.
\end{theorem}

The idea is to apply Corollary~\ref{cor:local1}(iii), together with a theorem of Host \cite{Host2}, which implies that under our conditions, for Lebesgue-a.e.\ $\om$, the 
orbit of the point on the diagonal of the torus $\om\one$ under the endomorphism $\xi\mapsto \Sf_\zeta^{\sf T}\xi$ (mod $\Z^d$) equidistributes and then one can show that \eqref{cond1}
implies $\chi^+_{\zeta,\om\one} < \frac{\log\theta}{2}$ for a.e.\ $\om$. Recently R. Yaari, a Master's student from the Bar-Ilan University, found an alternative proof, using results from uniform distribution theory \cite{Yaari}. This allowed him to extend Theorem~\ref{thm1} to a class of reducible matrices and simultaneously obtain results
on singularity of substitution $\R$-actions.

For a practical verification of the condition (\ref{cond1}) one can use the following standard result, which follows from Kingman's Theorem, see \cite{Viana}.

\begin{lemma} \label{lem-vspom1}
Suppose that $\Sf_\zeta$ has no eigenvalues that are roots of unity, so that $\chi_\zeta$ is well-defined. Then
\be \label{cond2}
\chi_\zeta = \inf_k \frac{1}{k} \int_{\T^d} \log\|\Cc_{\zeta^k}(\xi)\|\,dm_d(\xi).
\ee
\end{lemma}


\subsection{Spectral measure at zero}

For test functions of mean zero we can get much more accurate estimates at $\om= 0$. They are closely related to the question of {\em deviation of ergodic averages}
$S_R^{(x,t)}(f,0) = \int_0^R (f\circ h_\tau)(x,t)\,d\tau$, see Lemma~\ref{lem-spec2}. Such questions for substitutions were studied by many authors, see e.g.,
\cite{Adam,DT,DKT}.
 By the Birkhoff Ergodic Theorem, $R^{-1} S_R^{(x,t)}(f,0)\to \int f\,d\wt\mu = 0$, so one should expect
a different scaling law. In \cite{BuSo13,BuSo14} we obtained asymptotic expansions and even a limit law under appropriate conditions, adapting the methods of 
Bufetov \cite{Bufetov1} developed for translation flows. Here we restrict ourselves to a simple corollary of Theorem~\ref{thm:local1}.

Let $\zeta$ be a primitive aperiodic substitution on $\A = \{1,\ldots,d\}$ with a non-singular substitution matrix $\Sf_\zeta$.
Let $\theta=\theta_1 >|\theta_2|\ge\cdots$ be the eigenvalues of $\Sf_\zeta$, in the order of descent by their absolute value. Suppose that $|\theta_2|>1$. 
Denote by $P_j$ the projection from $\R^d$ to the linear span of the generalized eigenvectors corresponding to $\theta_j$, commuting with $\Sf_\zeta^{\sf T}$.

\begin{corollary} \label{cor:zero}
Let $\vec s\in \R^d_+$ and consider the suspension flow $(\Xx_\zeta^{\vec{s}}, h_t, \wt \mu)$.
Let $f(x,t) = \sum_{j\in \A} \Onne_{[j]} \cdot \psi_j(t)$ be a Lip-cylindrical function and $\sig_f$ the corresponding spectral measure.
Let 
$$
k:= \min\bigl\{j\ge 1: P_j\vec{z} \ne 0\},\ \ \mbox{where}\ \ \vec z = (\what{\psi}_j(0))_{j\in \Ak}.
$$
If $|\theta_k|>1$, then
$$
\underline{d}(\sig_f,0) = 2 - \frac{2\log|\theta_k|}{\log\theta_1}\,.
$$
If $|\theta_k|\le 1$, then $\underline{d}(\sig_f,0)\ge 2$.
\end{corollary}

Before derivation, we note that $k=1$ if and only if $\int f\,d\wt\mu\ne 0$, in which case $\sig_f$ has a point mass at zero and hence $\underline{d}(\sig_f,0) = 0$. Indeed,
by the definition of cylindrical function,
$$
\int f\,d\wt\mu = \sum_{j=1}^d \int_{[0,s_j]} \psi_j\cdot \mu([j]) = \sum_{j=1}^d \what{\psi}_j(0)\cdot \mu([j]).
$$
Recall from \eqref{freqPF} 
that ${\bbe_1^*}:=(\mu([j]))_{j\le d}$ is the PF eigenvector of $\Sf_\zeta$, so $P_1$, the projection onto the line of the Perron-Frobenius  eigenvector of $\Sf_\zeta^{\sf T}$, is given by 
$\langle \cdot, \bbe_1^*\rangle$, hence 
 $\int f\,d\wt\mu=0$ if and only if $P_1(\vec z) = 0$.
 
 \begin{proof}[Proof of Corollary~\ref{cor:zero}]
 By Theorem~\ref{thm:local1}, we only need to compute the Lyapunov exponent $\chi^+_{\zeta,0,\vec z}$. However, recall that at $\xi=0$ the spectral cocycle coincides with the transpose substitution
 matrix $\Sf_\zeta^{\sf T}$, and so $\chi^+_{\zeta,0,\vec z} = \log|\theta_k|$ follows by elementary linear algebra.
 \end{proof}


\subsection{Singular spectrum in the constant length case}

Now suppose that $\zeta$ is a primitive aperiodic substitution of constant length $q$.
Dekking's coincidence condition \cite{Dek1}, see Theorem~\ref{thm:dekking},
  gives an answer for when the spectrum is purely discrete. Thus it remains to analyze the continuous component, when it is present. The work of Queff\'elec \cite{Queff} contains a detailed and extensive study of this question; in particular, she expressed the maximal spectral type in terms of
generalized matrix Riesz products, which sometimes reduce to scalar generalized Riesz products. 
It is an {\em open problem} whether every {\em bijective} substitution (see Definition~\ref{def:bijective}) dynamical system is purely singular.
{Pure singular spectrum is known for Abelian bijective substitutions (this was outlined in \cite{Queff} and formally proved in \cite{Bartlett}; see \cite{BG2M} for a different proof).}
For bijective substitutions on two symbols we sketched a proof in Proposition~\ref{prop:singular}, based on the proofs from \cite{BGG,BG1}.
On the other hand, there are substitutions of constant length with a Lebesgue spectral component, such as the  Rudin-Shapiro substitution, see \cite{Queff}, and its generalizations, see \cite{AlLi,Frank,ChGr2}.
Recently Bartlett \cite{Bartlett} reworked (and in the case of one important example corrected) a part of Queff\'elec's theory, extended it to higher-dimensional block-substitutions of ``constant shape,'' and developed an algorithm for checking singularity. However, it proceeds by a ``case-by-case'' analysis and does not provide a general criterion. The following sufficient condition for singularity was obtained recently.

\begin{theorem}[Berlinkov and Solomyak \cite{BerSol}] \label{th-BeSo}
Let $\zeta$ be a primitive aperiodic substitution of constant length $q$. If  the substitution matrix $\Sf_\zeta$ has no eigenvalue whose absolute value equals $\sqrt{q}$, then the maximal spectral type of the substitution measure-preserving system 
$(X_\zeta,T,\mu)$ is singular.
\end{theorem}

\begin{remark} {\em
This condition for singularity is not necessary, see Chan and Grimm \cite{ChGr1} and \cite{BaGr}.
}
\end{remark}

In the rest of the section we describe the idea of the proof.
The key fact which distinguishes constant-length substitutions is that the ``renormalization'', or ``substitution'' action on $X_\zeta$ is closely related  to the times-$q$ map on the torus: $S_q\om\mapsto q\om$ (mod 1) on $\R/\Z\equiv [0,1]$, where $q$
is the length of the substitution. A measure on the torus is called $q$-mixing if it is invariant and mixing with respect to the times-$q$ map. By the Birkhoff Ergodic Theorem, any $q$-mixing measure is either singular or coincides with the Lebesgue (Haar) measure on the torus.
Further, let
\be \label{eq-convo}
\eta = \sum_{n=1}^\infty 2^{-n}\eta_{q^n},
\ee
where $\eta_{q^n}$ is the Haar measure on the finite group generated by $1/q^n$ on the torus.
Queffelec proved in \cite[Th.\,10.2 and Th.\,11.1]{Queff} that the maximal spectral type of the substitution automorphism is given by
$$
\wtil{\lam} = \lam * \eta,
$$
where $\lam = \lam_0 + \cdots +\lam_{k-1}$, with every $\lam_j$ being $q$-mixing, and moreover, $\lam_j$ is a linear  combination of correlation measures $\sig_{ab},\ a,b\in \Ak$, with $\lam_0 = \delta_0$. It follows that in order to prove singularity we only need to rule out that any of the $\lam_j$, $j=1,\ldots,k-1$, is the Lebesgue measure.
(In fact, the invariance of the matrix of correlation measures $\Sig_\zeta$ under $S_q$ follows from Lemma~\ref{lem-Riesz}; the measures $\lam_j$ arise from the
diagonalization of $\Sig_\zeta$.)

It was shown in \cite{BerSol} that every $\lam_j$, $j\ge 1$, can be expressed as a {\em positive} linear combination of spectral measures corresponding to cylindrical functions orthogonal to the constants. By Lemma~\ref{lem-equivalence}, instead of the substitution $\Z$-action we can work with the  the suspension flow with the
constant-one roof function. Suppose that one of $\lam_j$, $j\ge 1$, is Lebesgue.
Since the local dimension of the Lebesgue measure is equal to one everywhere, it follows that there exists a cylindrical $f\perp \Onne$ such that
$\underline{d}(\sig_f,0) = 1$. By Corollary~\ref{cor:zero}, this implies that there exists an eigenvalue $\theta_k$ of $\Sf_\zeta$ such that $|\theta_k| = \sqrt{\theta_1} = \sqrt{q}$, as desired. \qed


\subsection{The Fourier matrix cocycle  and the diffraction spectrum}
Consider the {\em self-similar} suspension flow over a substitution, when the vector $\vec s = \bbe_1$, defining the roof function,
satisfies $\Sf_\zeta^{\sf T} \bbe_1 = \theta \bbe_1$, that is,  the Perron-Frobenius eigenvector for  $\Sf_\zeta^{\sf T}$. 
Then we get
\be \label{coc2}
\Cc_\zeta(\om \bbe_1,n):= \Cc_\zeta\bigl(\theta^{n-1}\om \bbe_1\bigr)\cdot \ldots \cdot \Cc_\zeta(\theta\,\om\bbe_1)\cdot \Cc_\zeta(\om \bbe_1).
\ee
Baake and collaborators  \cite{BFGR,BGM,BG2M} 
considered  (\ref{coc2}) as a cocycle on $\R$ over the infinite-measure preserving action $\om\mapsto \theta \om$, which they called the {\em Fourier matrix cocycle}. 
They studied the diffraction spectrum of substitution systems in the self-similar case, which is closely related to the dynamical spectrum, see Section~\ref{sec:diffract}.
In particular, singularity of the maximal spectral type implies singularity of
the diffraction spectrum measure. 
The diffraction measure is represented as a matrix Riesz product, in the spirit of
Lemma~\ref{lem-Riesz} for substitution $\Z$-actions.
The starting point in their approach
is the idea of Kakutani, described in Section 3.2, namely, to analyze the renormalization relations satisfied by the absolutely continuous component of the diffraction measure.
Baake, Grimm, and Ma\~nibo \cite[Theorem 3.28]{BG2M} proved that if the upper Lyapunov exponent satisfies a condition, which in our notation can be stated as 
$\chi^+_{\zeta,\om\bbe_1}< \half\log\theta-\eps$, for some $\eps>0$, for Lebesgue-a.e.\ $\om\in \R$, then the absolute component must be trivial, and hence
the diffraction measure is singular. Effective numerical methods allowed the authors of \cite{BG2M} to verify this in many examples.

In fact, this approach was first developed by  Baake, Frank, Grimm, and Robinson \cite{BFGR} for a specific 
non-Pisot substitution on a two-letter alphabet $0\to 0111,\ 1\to 0$. Later
Baake, Grimm, and Ma\~nibo \cite{BGM}  extended the result on singularity of the spectrum
 to the family of substitutions $0\to 01^k,\ 1\to 0$, with $k\ge 4$. Then Baake, G\"ahler, Grimm, and
 Ma\~nibo \cite{BG2M} considered general primitive substitutions and extended the method to higher-dimensional self-similar tiling substitutions.


\section{Regularity of spectral measures}

Whereas the results of the previous section say something about the local dimension of spectral measures at almost every point, much more subtle is the question of {\em uniform}  
 local dimension bounds. This question is closely related to {\em quantitative rates of weak mixing}. First, a few words of motivation. If the substitution flow
has a non-trivial eigenvalue $\om\in\R$, then the spectral measure of a test function (not orthogonal to  the corresponding eigenfunction) has a point mass at $\om$ and
hence no quantitative bound is possible in a neighborhood of $\om$; the local dimension is zero. 
For flows that are known to be weakly mixing, there are no point masses, and it 
is natural to ask whether one  can get quantitative bounds for the modulus of continuity, ideally, H\"older bounds of the form $\sig_f(B_r(\om)) \le Cr^\alpha$. To have any hope
for such bounds to be uniform, we either need to assume that $f$ is orthogonal to constants, in order to avoid the trivial point mass at zero, or take $|\om|> \delta>0$ for
some fixed $\delta$.

In view of Lemma~\ref{lem-spec2} and \eqref{ki2}, we need to estimate from  above the norm of the matrix product appearing in the formula for the spectral cocycle
\eqref{cocycle3}. First, a definition.

\begin{defi}
A word $v\in \A^+$ is called a {\em return word} for the substitution $\zeta$ if $v$ separates two successive occurrences of the same letter in elements of $X_\zeta$; more precisely, $vc$ occurs in some
(and then all) $x\in X_\zeta$, where $c$ is the first letter of $v$.  We say that $v$ is a {\em good return word} for the substitution if $vc$ occurs in every word $\zeta(b),\ b\in \A$.
\end{defi}

We can always pass from $\zeta$ to $\zeta^k$ without affecting the substitution space, and thus we can assume that all return words are good. Below we write $\|x\|_{_{\R/\Z}}=
\dist(x,\Z)$ for $x\in \R$.

\begin{prop}[{\cite[Prop.\,4.4]{BuSo14}}] \label{prop-cocy}
Let $\zeta$ be a primitive aperiodic substitution on $\A$ and $v$ a good return word for $\zeta$. Let $\theta$ be the Perron-Frobenius eigenvalue for $\Sf_\zeta$ and $\vec s\in \R^d_+$.
Then there exist $c_1\in (0,1)$ and $C>0$ depending only on the 
substitution $\zeta$ and on $\vec s$, such that for any  $n\in \N$ and $\om\in \R$, 
\be \label{cu3}
\bigl\|\Cc_\zeta(\om\vec s, n)\bigr\| \le C \theta^n \prod_{k=0}^{n-1} \Bigl( 1 - c_1 \bigl\|\om|\zeta^k(v)|_{\vec{s}}\bigr\|_{_{\R/\Z}}^2\Bigr).
\ee
\end{prop}

\begin{proof}[Idea of the proof of the proposition]
By the definition of a good return word, we can write $\zeta(b) = p^b vc q^b$, where $p^b$ and $q^b$ are words, possibly empty, and $v$ starts with $c$. By \eqref{cu2}, the expression for $\bigl[\Cc_\zeta\bigl( (\Sf_\zeta^{\sf T})^k \om\vec s\bigr)\bigr]_{(b,c)}$ includes the terms
$e^{-2\pi i \om |\zeta^k(p^b)|_{\vec{s}}}$ and $e^{-2\pi i \om |\zeta^k(p^b v)|_{\vec{s}}}$, hence
\begin{eqnarray*}
\Bigl|\bigl[\Cc_\zeta\bigl( (\Sf_\zeta^{\sf T})^k \om\vec s\bigr)\bigr]_{(b,c)}\Bigr| & \le & \Sf^{\sf T}_\zeta(b,c) - 2 + \bigl|1 + e^{-2\pi i \om|\zeta^k(v)|_{\vec s}}\bigr| \\
& \le &  \Sf^{\sf T}_\zeta(b,c) - \textstyle{\half} \bigl\|\om|\zeta^k(v)|_{\vec{s}}\bigr\|_{_{\R/\Z}}^2, 
\end{eqnarray*}
where we used an elementary inequality $|1 + e^{2\pi i \tau}| \le 2 - \half\|\tau\|_{_{\R/\Z}}^2,\ \tau\in \R$. Then one can show that multiplying the product of matrices by
$\Cc_\zeta\bigl( (\Sf_\zeta^{\sf T})^k \om\vec s\bigr)$, we ``gain'' in an upper bound a factor of $1 - c_1 \bigl\|\om|\zeta^k(v)|_{\vec{s}}\bigr\|_{_{\R/\Z}}^2$, compared to 
the multiplication by $\Sf_\zeta^{\sf T}$. Using that $\|(\Sf_\zeta^{\sf T})^n\| \sim \theta^n$ yields the bound \eqref{cu3}. For more details, see \cite[Prop.\,3.5 and 4.4]{BuSo14}.
\end{proof}

\smallskip

The formula \eqref{cu3} suggests that Diophantine approximation issues may play a role. This is especially transparent in the special case of self-similar substitution flow,
when $\vec s = \bbe_1$ satisfies $\Sf_\zeta^{\sf T} \bbe_1 = \theta \bbe_1$, that is,  the Perron-Frobenius eigenvector for  $\Sf_\zeta^{\sf T}$. By \eqref{Der}, we then have 
$|\zeta^k(v)|_{\bbe_1} = \theta^k |v|_{\bbe_1}$ and \eqref{cu3} becomes
\be \label{Der4}
\bigl\|\Cc_\zeta(\om\bbe_1, n)\bigr\| \le C \theta^n \prod_{k=0}^{n-1} \bigl( 1 - c_1  \bigl\|\om |v |_{\bbe_1}\cdot \theta^k \bigr\|_{_{\R/\Z}}^2\bigr).
\ee
Recall the classical results, mostly due to Pisot, and a famous open problem, see \cite{Salem_book}:

\begin{itemize}
\item The set $\Sk:= \{\theta>1:\ \exists\, x>0$ such that  $\|\theta^k x\|_{_{\R/\Z}} \to 0\}$ contains the set of Pisot numbers. For Pisot $\theta$ holds $\|\theta^k x\|_{_{\R/\Z}}\le C\rho^k$ for 
$x\in \Z[\theta]$, with  $\rho\in (0,1)$ (the maximal modulus of a Galois conjugate).
\item If $\theta\in \Sk$ is algebraic or if $\sum_k \|\theta^k x\|_{_{\R/\Z}}^2 < \infty$ for $\theta>1$ and $x>0$, then $\theta$ is Pisot.
\item It is an {\bf open problem} whether $\Sk$ contains any transcendental numbers; however, Salem showed that $\Sk$ is countable.
\end{itemize}

Thus, \eqref{Der4} does not give anything useful (at least, uniformly in $\om$) when $\theta$ is Pisot, which is not surprising since then the substitution flow has
a large discrete component. In fact, the following holds.

\begin{theorem}[{\cite{SolTil}}] \label{th-Pisot}
Given a primitive aperiodic substitution $\zeta$, the self-similar substitution flow $(\Xx_\zeta^{\bbe_1}, h_t,\wt\mu)$ is weakly mixing if and only if $\theta$ is not
a Pisot number.
\end{theorem}

\subsection{Digression: Bernoulli convolutions}
Another area where Pisot numbers appear, is the study of {\em infinite Bernoulli convolutions}. 
For $\lam\in (0,1)$, the {\em classical} infinite Bernoulli convolution measure $\nu_\lam$ is defined as the distribution of the random
series $\sum_{n=0}^\infty \pm \lam^n$, where the signs are chosen i.i.d., with probabilities $(\half,\half)$. The  {\em biased} Bernoulli convolution $\nula^p$ is defined the same way, but the probabilities of the signs are $(p,1-p)$ for some $p\ne \half$. For $\lam< \half$ the Bernoulli convolution $\nula$ is the Cantor-Lebesgue measure on a self-similar Cantor set 
of contraction ratio $\lam$, but for $\lam>\half$ it is a self-similar measure with ``overlaps'' which makes its analysis difficult.

By the ``pure types'' law, $\nu_\lam^p$ is either singular or absolutely continuous.
The Fourier transform of $\nu_\lam^p$ is given by
$$
\widehat{\nu}_\lam^p(\xi) = \prod_{n=0}^\infty \bigl(pe^{-2\pi i \lam^n \xi} + (1-p) e^{2\pi i \lam^n \xi}\bigr)
$$
Erd\H{o}s \cite{Erd0} proved that $\widehat{\nu}_\lam^p(\xi)\not\to 0$ as $\xi\to\infty$ when $\lam^{-1}$ is a Pisot number, hence such $\nu^p_\lam$ are singular. Salem \cite{Salem} showed that this characterizes Pisot numbers, namely, for all other $\lam$ the Fourier transform vanishes at infinity. However, Salem's
 proof does not yield any quantitative estimates of the decay of the Fourier transform. Using a combinatorial argument, essentially a ``quantitative upgrade'' of Salem's proof that $\Sk$ is
 countable, Erd\H{o}s
\cite{Erd} proved that for almost every $\lam$ the Fourier transform has some power decay, which he then used to show that $\nu^p_\lam$ is absolutely continuous for almost every $\lam$ sufficiently close to one. (Erd\H{o}s considered only  $p=\half$, but the same proof works for an arbitrary $p\in (0,\half)$.) Kahane \cite{Kahane} observed that Erd\H{o}s' proof actually gives an estimate of the dimension of the exceptional set, and in fact the power decay of the Fourier transform (with {\em some} power depending on $\lam$ and $p$) holds for all $\lam$ outside a set of Hausdorff dimension zero. This technique came to be 
known as the ``Erd\H{o}s-Kahane argument''. However, their method is ``generic'' and does not provide a single specific example.
As far as concrete examples go,
Garsia \cite{Garsia} proved that $\nula$ is absolutely continuous, with an $L^\infty$ density, when $\theta = \lam^{-1}$ is an algebraic integer, all of whose conjugates lie outside of
the unit circle and
the constant term of the minimal polynomial is $\pm 2$. (Here it is important that the Bernoulli convolution is unbiased!) 
Very recently, Yu \cite{Yu} showed that for these ``Garsia numbers'' the density is even continuous.
While dramatic progress in our understanding of Bernoulli convolutions occurred in the last decade,
 it is still an open problem whether there are any non-Pisot parameters $\theta \in (1,2)$ for which $\nu_{1/\theta}$ is singular. As of this writing, the record for a ``typical'' parameter belongs to Shmerkin \cite{Shmerkin}, who proved, in particular, that for all $\lam\in (\half,1)$ outside of a set of zero Hausdorff  dimension, $\nula$ is absolutely continuous with a density in $L^q$ for all $q<\infty$. See the survey by Peres, Schlag and Solomyak \cite{PSS}
 on the status of the problem around 2000 and that of Varj\'u \cite{Varju} on the progress until 2018.
 
 Another open problem is to determine for which $\lam$ and $p$ the Fourier transform of the measure $\nula^p$ has a power decay at infinity. 
 The only known case is again Garsia numbers, and only 
 for $p=\half$, due to Dai, Feng, and Wang \cite{DFW}.
 
 An elementary inequality, as in \eqref{Der4}, yields
 $$
 |\widehat{\nu}_{1/\theta}^p(\xi) | \le \prod_{n=0}^N \Bigl( 1 - {\frac{1-p}{2}}\|2\theta^n \zeta\|_{_{\R/\Z}}^2\Bigr),\ \ \mbox{for}\ \zeta = \theta^{-N} \xi,\ \xi \in [\theta^{-N},\theta^{N+1}].
 $$
 We can clearly see a connection with \eqref{Der4}. In fact, if $\what{\nu}_\lam^p$ for $p\ne \half$,  
 has a power decay at infinity for $\lam = \theta^{-1}$, where $\theta$ is from \eqref{Der4}, then the spectral measures of the  self-similar suspension flow $(\Xx_\zeta^{\bbe_1}, h_t,\wt\mu)$
 are H\"older continuous. This is still unknown, but we were able to obtain a weaker, power of a logarithm, bound in \cite[Cor.\,A.3 and Th.\,5.1]{BuSo14}.
 
 \begin{theorem}[{\cite{BuSo14}}] \label{th-log}
 
 {\rm (i)} Let $\theta>1$ be an algebraic integer which has at least one Galois conjugate outside the unit circle, and let $\lam = \theta^{-1}$. Then for any $p\in (0,1)$ there exists $\alpha>0$ such that
$$
\sup_{\xi\in \R} |\widehat{\nu}^p_\lam(\xi)| (\log(2+|\xi|))^\alpha < \infty.
$$

{\rm (ii)} Let $\zeta$ be an aperiodic primitive substitution whose substitution matrix has the PF eigenvalue $\theta$ as in part (i). Consider the self-similar suspension flow $(\Xx_\zeta^{\bbe_1}, h_t,\wt\mu)$. Then there exists $\gam>0$ such that for any $B>1$, there exists $C_B>0$ such that for any simple cylindrical function $\phi$ the spectral measure $\sig_\phi$ satisfies
$$
\sig_\phi(B_r(\om)) \le C_B \bigl(\log (1/r)\bigr)^{-\gam},\ \ \mbox{for all}\ |\om|\in [B^{-1},B]\ \mbox{and}\ r>0.
$$
 \end{theorem}


\subsection{H\"older continuity of spectral measures for almost every substitution  flow} As it often happens, stronger results may be obtained for a ``generic'' or ``typical'' system.

Proposition~\ref{prop-cocy} may be compared with a theorem of Clark and Sadun:

\begin{theorem}[{Clark and Sadun \cite[Th.\,2.3]{CS03}}] \label{th-CS}
A number $\om\in \R$ is in the point spectrum of $(\Xx_\zeta^{\vec{s}}, h_t, \wt \mu)$ if and only if, for every return word $v$,
$$
\bigl\|\om|\zeta^k(v)|_{\vec{s}}\bigr\|_{_{\R/\Z}}\to 0\, \ \mbox{as}\ k\to \infty.
$$
\end{theorem}

Recall that $\bigl| \zeta^k(v) \bigr|_{\vec s}= \bigl\langle \vec{\ell}(v), (\Sf_\zeta^{\sf T})^k {\vec s}\bigr\rangle$, see \eqref{Der}. To continue, let us assume that the
$\Z$-module generated by the population vectors of return words coincides with $\Z^d$. (In reality, for our argument it suffices that this module is a full rank lattice, which is
guaranteed, for instance, when the characteristic polynomial of $\Sf_\zeta$ is irreducible.) In this case, the criterion from Theorem~\ref{th-CS} can 
be rewritten as follows:
\be \label{cond-eig}
\mbox{$\om$ is an eigenvalue for $(\Xx_\zeta^{\vec{s}}, h_t, \wt \mu)$}\ \iff\ \|\om(\Sf_\zeta^{\sf T})^k {\vec s}\|_{\R^d/\Z^d}\to 0,\ k\to \infty.
\ee
Here $\|\cdot\|_{\R^d/\Z^d}$ denotes the distance to the nearest integer lattice point. It is reasonable to normalize the vector of heights $\vec s$, say, by requiring that 
$\vec s \in \Delta^{d-1}:= \{\vec s\in \R^d_+:\ \sum_{j=1}^d s_j =1\}$. Indeed, simply rescaling $\vec s$ results in the same dynamical system, just with rescaled
time. We see from \eqref{cond-eig} that the {\em stable manifold} 
$$
E^s := \bigl\{\xi\in \T^d:\ \lim_{k\to \infty} \|(\Sf_\zeta^{\sf T})^k\xi \|_{\R^d/\Z^d} = 0 \bigr\}
$$
for the endomorphism $\xi \mapsto \Sf_\zeta^{\sf T}$ of the torus $\T^d$ 
plays a role. To continue the discussion, let us assume that 
$|\det(\Sf_\zeta)| =1$ for simplicity, then $\Sf_\zeta^{\sf T}$ induces a toral automorphism and we don't have to worry about points with rational coordinates.
Let $q_-$ be the number of 
eigenvalues of $\Sf_\zeta$ of modulus strictly less than 1. Then $E^s$ is a $q_-$-dimensional linear subspace of $\R^d$ modulo $\Z^d$. The flow $(\Xx_\zeta^{\vec{s}}, h_t, \wt \mu)$ is not weakly mixing, i.e., has non-trivial point spectrum if and only if the line $\Span_\R(\vec s)$ intersects $E^s$ non-trivially. Thus, weak mixing holds for 
suspension flows with $\vec s\in \Delta^{d-1}$ in the complement of a set of Hausdorff dimension $q_-$. The extreme cases are: 
\begin{itemize}
\item
$q_-=d-1$; this is the irreducible Pisot case, the
system is never weakly mixing;
\item
$q_-=0$; the system is always weakly mixing.
\end{itemize}

Returning to Proposition~\ref{prop-cocy} and assuming that the $\Z$-module generated by the population vectors of {\em good} return words coincides with $\Z^d$, we
obtain from \eqref{cu3}:
$$
\bigl\|\Cc_\zeta(\om\vec s, n)\bigr\| \le C \theta^n \prod_{k=0}^{n-1} \bigl( 1 - c_1 \|\om(\Sf_\zeta^{\sf T})^k {\vec s}\|^2_{\R^d/\Z^d}\bigr).
$$
Let $B>1$. In view of this and the results of Section 4, uniform H\"older continuity of Lip-cylindrical spectral measures for $\om \in [B^{-1}, B]$ will hold, provided there exist
$\delta>0$ and $\rho>0$ such that 
$$
\liminf_{N\to \infty} \frac{1}{N} \left|\left\{k\in \{1,\ldots,N\}: \exists\,\om \in [B^{-1},B],\ \|\om(\Sf_\zeta^{\sf T})^k {\vec s}\|_{\R^d/\Z^d} \ge \rho\right\} \right|> \delta.
$$
On this route, using a version of the ``Erd\H{o}s-Kahane argument'', we obtained the following result. This is a combination of \cite[Theorem 4.2]{BuSo14} and
\cite[Theorem 1.2]{BuSo18}.

\begin{theorem}
Let $\zeta$ be a primitive aperiodic substitution on $d\ge 2$ symbols, with substitution matrix $\Sf_\zeta$. Suppose that the characteristic polynomial of $\Sf_\zeta$ is
irreducible and there are $q_+\ge 2$ eigenvalues of $\Sf_\zeta$ of modulus strictly greater than one. Then for any $\eps>0$ there exists $\gam >0$ depending
only on the substitution and on $\eps$ and an exceptional set $\Ek_\eps\subset \Delta^{d-1}$ of Hausdorff dimension at most $d-q_+ + \eps$ such that for every 
$\vec s\in \Delta^{d-1}\setminus \Ek_\eps$ there exists $C(\vec s)>0$ , such that for every Lip-cylindrical function $f$ of mean zero,
$$
\sig_f(B_r(\om)) \le C r^\gam,\ \ \mbox{for all}\ \om\in \R\ \mbox{and}\ r>0.
$$
\end{theorem}

Observe that the estimate of the dimension of the exceptional set is sharp in the case when $\Sf_\zeta$ has no eigenvalues on the unit circle, in view of Theorem~\ref{th-CS}.



\end{document}